%% file: orbits.tex
\documentclass[12pt,a4paper,draft]{amsart}
\input{pream}\input{defin}
\newcommand{\IC}{\operatorname{IC}}\newcommand{\oRB}{\OOo\,}\newcommand{\cloRB}{\overline{\OOo}\,}
\newcommand{\Uind}{\operatorname{Uind}}
\begin{document}

\title[Norm closures of orbits of bounded operators]{Norm closures of orbits\\of bounded operators}
\myData
\begin{abstract}
To every bounded linear operator $A$ between Hilbert spaces $\HHh$ and $\KKk$ three cardinals $\iota_r(A)$,
$\iota_i(A)$ and $\iota_f(A)$ and a binary number $\iota_b(A)$ are assigned in terms of which the descriptions
of the norm closures of the orbits $\{G A L^{-1}\dd\ L \in \GGg_1,\ G \in \GGg_2\}$ are given for $\GGg_1$ and $\GGg_2$
(chosen independently) being the trivial group, the unitary group or the group of all invertible operators on $\HHh$
and $\KKk$, respectively.
\end{abstract}
\maketitle


\SECT{Introduction}

For a Hilbert space $\HHh$, denote by $\GGg(\HHh)$ and $\UUu(\HHh)$ the group of all isomorphisms
(i.e. linear homeomorphisms) of $\HHh$ and the unitary group of $\HHh$, respectively. Additionally, let $I_{\HHh}$
be the identity operator on $\HHh$.
For study of the geometry of the Banach space $\BBb(\HHh,\KKk)$ of all bounded linear operators from the Hilbert space
$\HHh$ into a Hilbert space $\KKk$ the natural action $(\GGg(\HHh) \times \GGg(\KKk)) \times \BBb(\HHh,\KKk) \ni ((L,G),X)
\mapsto G X L^{-1} \in \BBb(\HHh,\KKk)$ of the group $\GGg(\HHh) \times \GGg(\KKk)$ plays an important role.
Especially the literature concerning the orbits (and their closures) under this action of closed range operators is still
growing up (see e.g. \cite{corach} and references there). This includes the theory of Fredholm and semi-Fredholm operators,
for which the index `$\ind$' is naturally defined and well behaves. However, most of results on closed range operators
is settled in a separable (infinite-dimensional) Hilbert space. Also hardly ever operators with nonclosed ranges
are considered when speaking about indices or orbits under the group action. The aim of the paper is to fill this lack
and give a full answer (see \THM{orbit}) to the following problem:
\begin{quote}
\textit{Given an operator $A \in \BBb(\HHh,\KKk)$, describe the closure of the orbit of $A$ under the natural action
of $\GGg(\HHh) \times \GGg(\KKk)$.}
\end{quote}
We shall also solve analogous problems for group orbits with respect to the following subgroups of $\GGg(\HHh) \times
\GGg(\KKk)$ (under the same action): $\GGg(\HHh) \times \{I_{\KKk}\}$ and $\{I_{\HHh}\} \times \GGg(\KKk)$ (\THM{G}),
$\UUu(\HHh) \times \UUu(\KKk)$ (\PRO{UU}), $\UUu(\HHh) \times \{I_{\HHh}\}$ and $\{I_{\HHh}\} \times \UUu(\KKk)$ (\THM{U};
note that the orbit of $A$ with respect to the groups $\GGg(\HHh) \times \UUu(\KKk)$, $\UUu(\HHh) \times \GGg(\KKk)$
and $\GGg(\HHh) \times \GGg(\KKk)$ coincide, which follows e.g. from Theorem 3.1 of \cite{ranges}---see \PRO{orbit}).
The main results on these are settled in any Hilbert spaces (without restriction on dimensions) and deal with any bounded
operator.\par
In comparison to the characterization of the members of the orbit of an operator under the action of $\GGg(\HHh) \times
\GGg(\KKk)$ which highly depends on the geometry of the range of the operator (see Theorem 3.4 of \cite{ranges}
or \PRO{orbit} in Section 3 below), the description of the closure of this orbit is given only in terms of four indices
(that is, cardinal numbers), which seems to be surprising. Among many applications one may find: generalizations
of results of Izumino and Kato \cite{i-k} on the closure of $\GGg(\HHh)$ and of Mbekhta \cite{mbekhta} on the boundaries
of sets of semi-Fredholm operators of arbitrarily fixed index, an extension of the notion of the index `$\ind$'
and characterization of all closed two-sided ideals of $\BBb(\HHh)$ for nonseparable $\HHh$.\par
\textbf{Notation.} In this paper $\HHh$ and $\KKk$ denote (complex) Hilbert spaces. $\BBb(\HHh,\KKk)$ is the Banach space
of all bounded operators from $\HHh$ into $\KKk$; $\GGg(\HHh,\KKk)$ and $\UUu(\HHh,\KKk)$ are, respectively,
the set of all isomorphisms and unitary operators from $\HHh$ onto $\KKk$. When $\KKk = \HHh$, we write $\BBb(\HHh)$,
$\GGg(\HHh)$ and $\UUu(\HHh)$ instead of $\BBb(\HHh,\HHh)$, $\GGg(\HHh,\HHh)$ and $\UUu(\HHh,\HHh)$. Additionally,
$\BBb_+(\HHh)$ stands for the set of all nonnegative (bounded) operators on $\HHh$. Whenever $V$ is a closed subspace
of $\HHh$, $\HHh \ominus V$ and $P_V \in \BBb(\HHh)$ denote the orthogonal complement of $V$ in $\HHh$ and the orthogonal
projection onto $V$. $I_{\HHh}$ is the identity operator on $\HHh$ and $\dim \HHh$ is the dimension of $\HHh$ as a Hilbert
space (i.e. $\dim \HHh$ is the power of any orthonormal basis of $\HHh$). For $A \in \BBb(\HHh,\KKk)$, $\NnN(A)$, $\RrR(A)$
and $\overline{\RrR}(A)$ denote the kernel, the range and the closure of the range of $A$. The \textit{polar decomposition
of $A$} has the form $A = Q |A|$ where $|A| := \sqrt{A^*A}$ and $Q$ is a partial isometry such that $\NnN(Q) = \NnN(A)$.
Whenever we speak about convergence, closures, open sets, etc., in $\BBb(\HHh,\KKk)$, all they are understood in the norm
topology. By $\Bb(\RRR_+)$ we denote the $\sigma$-algebra of all Borel subsets of $\RRR_+ := [0,+\infty)$.

\SECT{Operator ranges}

In this part we recall the characterization of operator ranges and we define two auxiliary indices of such spaces
which will find an application in the sequel.\par
Whenever $\EEe$ is a pre-Hilbert space, $\bar{\EEe}$ stands for its completion. A pre-Hilbert space $\EEe$ is said to be
an \textit{operator range} iff there is a Hilbert space $\HHh$ and a bounded operator $T\dd \HHh \to \bar{\EEe}$ such
that $\RrR(T) = \EEe$.\par
Whenever $\HHh_1, \HHh_2,\ldots$ is a sequence of mutually orthogonal closed subspaces of a Hilbert space $\HHh$,
let us denote by $\SsS(\HHh_1,\HHh_2,\ldots)$ the linear subspace of $\HHh$ consisting of all vectors $x \in \HHh$
of the form $x = \sum_{n=1}^{\infty} x_n$ where $x_n \in \HHh_n$ and $\sum_{n=1}^{\infty} 4^n \|x_n\|^2 < +\infty$.\par
A fundamental result on operator ranges is the following

\begin{thm}{range}
For a pre-Hilbert space $\EEe$ \tfcae
\begin{enumerate}[\upshape(a)]
\item $\EEe$ is an operator range,
\item there is a sequence $\HHh_1,\HHh_2,\ldots$ of mutually orthogonal closed subspaces of $\bar{\EEe}$
   such that $\EEe = \SsS(\HHh_1,\HHh_2,\ldots)$.
\end{enumerate}
\end{thm}

For proof, see e.g. \cite{ranges}.

\begin{dfn}{index}
Let $\EEe$ be a pre-Hilbert space. The cardinal
$$
\IC(\EEe) = \min \{\dim (\bar{\EEe} \ominus V)\dd\ V \textup{ is a complete subspace of $\EEe$}\}
$$
is called the \textit{index of incompleteness of $\EEe$}. The \textit{binary index of $\EEe$}, in symbol $\bB(\EEe)$,
is defined as follows: $\bB(\EEe) = 1$ iff $\EEe$ contains a (necessarily complete) subspace isomorphic to $\bar{\EEe}$;
otherwise $\bB(\EEe) = 0$.
\end{dfn}

The following result shows how useful are just defined indices. Its proof is left as an exercise (the points (d) and (e)
of it follow from \THM{range}).

\begin{pro}{indices}
Let $\EEe$ be a pre-Hilbert space.
\begin{enumerate}[\upshape(a)]
\item $\EEe$ is complete iff $\IC(\EEe) = 0$. If $\EEe$ is incomplete, $\IC(\EEe)$ is infinite.
\item If $\IC(\EEe) < \dim \bar{\EEe}$, then $\bB(\EEe) = 1$.
\item If $\bB(\EEe) = 1$, there is a complete subspace $V$ of $\EEe$ such that $\dim V = \dim \bar{\EEe}$
   and $\dim (\bar{\EEe} \ominus V) = \IC(\EEe)$.
\item If $\EEe$ is a range space, then for each $\beta < \dim \bar{\EEe}$ there is a complete subspace of $\EEe$
   of dimension $\beta$.
\item If $\EEe$ is a range space and $\bB(\EEe) = 0$, then $\dim \bar{\EEe}$ is an (infinite) limit cardinal of countable
   cofinality.
\end{enumerate}
\end{pro}

It is clear that if $\EEe$ and $\EEe'$ are two linearly isometric pre-Hilbert spaces, then $\IC(\EEe) = \IC(\EEe')$
and $\bB(\EEe) = \bB(\EEe')$. This property combined with the well known fact that the ranges of a bounded operator
and its adjoint operator are linearly isometric yields

\begin{pro}{adjoint}
If $A$ is a bounded operator between two Hilbert spaces, then $\IC(\RrR(A)) = \IC(\RrR(A^*))$
and $\bB(\RrR(A)) = \bB(\RrR(A^*))$.
\end{pro}

With use of Theorem 3.3 of \cite{ranges} on linearly isometric operator ranges we now give formulas for both the indices
$\IC(\EEe)$ and $\bB(\EEe)$ in case $\EEe$ is an operator range.

\begin{pro}{IC,b}
Let $\EEe$ be an infinite-dimensional operator range and $\HHh_1,\HHh_2,\ldots$ be a sequence as in the point \textup{(b)}
of \textup{\THM{range}}; that is, $\EEe = \SsS(\HHh_1,\HHh_2,\ldots)$.
\begin{enumerate}[\upshape(a)]
\item There is $N \geqsl 1$ such that $IC(\EEe) = \sum_{n=N}^{\infty} \dim \HHh_n$. More precisely,
   $\IC(\EEe) = \min \{\sum_{n=m}^{\infty} \dim \HHh_n\dd\ m \geqsl 1\}$.
\item $\bB(\EEe) = 1$ iff $\dim \HHh_j = \dim \bar{\EEe}$ for some $j \geqsl 1$.
\end{enumerate}
\end{pro}
\begin{proof}
Let $V$ be a complete subspace of $\EEe$. Note that the assertions of both the points (a) and (b) follow
from the following property:
\begin{quote}
($\star$)\quad \begin{tabular}{l}there is $U \in \UUu(\bar{\EEe}\,)$ and $N \geqsl 1$ such that\\
$U(\EEe) = \EEe$ and $U(V) \subset \bigoplus_{n=1}^N \HHh_n$,\end{tabular}
\end{quote}
which we now prove. Let $\EEe' = \EEe \cap (\bar{\EEe} \ominus V)$. Observe that $\EEe'$ is an operator range
(one does not need Corollary~2 of Theorem~2.2 of \cite{ranges} to see this) and therefore
$\EEe' = \SsS(\HHh_1',\HHh_2',\ldots)$ for suitable spaces $\HHh_1',\HHh_2',\ldots$ But then
$\EEe = \SsS(V,\HHh_1',\HHh_2',\ldots) = \SsS(\HHh_1,\HHh_2,\ldots)$ and it follows from the proof of Theorem~3.3
of \cite{ranges} that there are $U$ and $N$ satisfying ($\star$).
\end{proof}

\SECT{Two-sided actions}

\begin{dfn}{orb}
Let $\HHh$ and $\KKk$ be Hilbert spaces and $\GGg_1$ and $\GGg_2$ be subgroups of $\GGg(\HHh)$ and $\GGg(\KKk)$
respectively. For $A \in \BBb(\HHh,\KKk)$ let $\oRB_{\GGg_1}^{\GGg_2}(A)$ be the orbit of $A$ with respect to the left
action of $\GGg_1 \times \GGg_2$ on $\BBb(\HHh,\KKk)$ given by $(\GGg_1 \times \GGg_2) \times \BBb(\HHh,\KKk) \ni
((G_1,G_2),X) \mapsto G_2 X G_1^{-1} \in \BBb(\HHh,\KKk)$; that is, $\oRB_{\GGg_1}^{\GGg_2}(A) = \{G_2 A G_1^{-1}\dd\
G_j \in \GGg_j\}$. The closure in $\BBb(\HHh,\KKk)$ of this orbit is denoted by $\cloRB_{\GGg_1}^{\GGg_2}(A)$.\par
When $\GGg_j$ coincides with the whole group of invertible operators or with the unitary group, then $\GGg_j$
in the notation $\oRB_{\GGg_1}^{\GGg_2}$ will be replaced by the letter $G$ or $U$, respectively. When $\GGg_j$ is
the trivial group (consisting only of the identity operator), $\GGg_j$ is omitted in the latter notation.
\end{dfn}

\noindent\textbf{Notation.} For $A \in \BBb(\HHh,\KKk)$ let $\Upsilon(A)$ constist of all closed linear subspaces $V$
of $\HHh$ for which there is a positive constant $c$ such that $\|Ax\| \geqsl c \|x\|$ for any $x \in V$. Observe that
$A(V)$ is closed in $\KKk$ for every $V \in \Upsilon(A)$. Additionally, to simplify further arguments, for each
$V \in \Upsilon(A)$ we use the following notation: $\xi_A(V) = (\dim (\HHh \ominus V),\dim V,\dim(\KKk \ominus A(V))$.
Put $\Lambda(A) = \{\xi_A(V)\dd\ V \in \Upsilon(A)\}$. Notice that $\Upsilon(A) = \Upsilon(|A|)$.\par

For completeness of the lecture, we begin with recalling the characterization of members of suitable orbits.

\begin{pro}{orbit}
Let $A \in \BBb(\HHh,\KKk)$. Let $\NNn_A$ be the set of all $B \in \BBb(\HHh,\KKk)$ such that
   $\dim \NnN(B) = \dim \NnN(A)$ and $\dim \NnN(B^*) = \dim \NnN(A^*)$.
\begin{enumerate}[\upshape(a)]
\item $\oRB_G^G(A)$ consists of all $B \in \NNn_A$ such that the ranges of $A$ and $B$ are isomorphic.
\item $\oRB_U^U(A)$ is the set of all $B \in \NNn_A$ such that $AA^*$ and $BB^*$ are unitarily equivalent.
\item $\oRB_G(A)$ consists of all $B \in \NNn_A$ such that $\RrR(B) = \RrR(A)$.
\item $\oRB_U(A)$ is the set of all $B \in \NNn_A$ such that $BB^* = AA^*$.
\item $\oRB_U^G(A) = \oRB_G^U(A) = \oRB_G^G(A)$, $\oRB^G(A) = \{B^*\dd\ B \in \oRB_G(A^*)\}$
   and $\oRB^U(A) = \{B^*\dd\ B \in \oRB_U(A^*)\}$.
\end{enumerate}
\end{pro}
\begin{proof}
The point (c) is Corollary 1 of Theorem 2.1 of \cite{ranges}; (a) and the first assertion of (e) is proved also
in \cite{ranges}, see Theorem 3.4 and its proof. To show the sufficiency of the conditions of (d), observe that if $Q_A$
and $Q_B$ are the partial isometries appearing in the polar decompositions of $A$ and $B \in \NNn_A$, respectively,
then $Q_A^*Q_B\bigr|_{\overline{\RrR}(B^*)}$ is extendable to a unitary $U \in \UUu(\HHh)$ which automatically satisfies
$B = AU$. Finally, notice that (b) basicly follows from (d) and the remainder of (e) is immediate.
\end{proof}

In \LEM{main}--(c) we shall show, independently of the foregoing result, that $\cloRB_G^G(A) = \cloRB_G^U(A)
= \cloRB_U^G(A)$ for each $A$.\par
The proof of the following easy result is omitted.

\begin{lem}{dim}
If $A \in \BBb(\HHh,\KKk)$, $W$ is a closed subspace of $\HHh$ such that $W \cap \NnN(A) = \{0\}$,
then $\dim (W \ominus V) = \dim (\overline{A(W)} \ominus A(V))$ for any space $V \in \Upsilon(A)$ contained in $W$.
\end{lem}

The next two results are our main tools.

\begin{pro}{inside}
For every $A \in \BBb(\HHh,\KKk)$ and each $\epsi > 0$ there is $V \in \Upsilon(A)$ such that
$|A|(V) = V$ and $\|A - A P_V\| \leqsl \epsi$.
\end{pro}
\begin{proof}
Let $E\dd \Bb(\RRR_+) \to \BBb(\HHh)$ be the spectral measure of $|A|$. It suffices to put $P = E([\epsi,+\infty))$
and $V = \RrR(P)$.
\end{proof}

A part of the point (a) of the following result is certainly known in perturbation theory. However, its short proof
is used to establish the remainder of (a) and therefore below we give full details.

\begin{lem}{main}
Let $A \in \BBb(\HHh,\KKk)$.
\begin{enumerate}[\upshape(a)]
\item If $A_1,A_2,\ldots \in \BBb(\HHh,\KKk)$ converge to $A$, then
   $$
   \Upsilon(A) \subset \bigcup_{n=1}^{\infty} \bigcap_{k=1}^{\infty} \Upsilon(A_k) \quad \textup{and} \quad
   \Lambda(A) \subset \bigcup_{n=1}^{\infty} \bigcap_{k=1}^{\infty} \Lambda(A_k).
   $$
   What is more, for each $V \in \Upsilon(A)$ there is $N \geqsl 1$ and a sequence $(Z_n)_{n=N}^{\infty} \in \UUu(\KKk)$
   such that $V \in \Upsilon(A_n)$, $Z_n(A(V)) = A_n(V)$ ($n \geqsl N$) and $Z_n P_{A(V)} \to P_{A(V)}\ (n \to \infty)$.
\item $\Lambda(G A L^{-1}) = \Lambda(A)$ for each $G \in \GGg(\KKk)$ and $L \in \GGg(\HHh)$.
\item $\cloRB_G^G(A) = \cloRB_G^U(A) = \cloRB_U^G(A)$ and $\cloRB_G^G(A)$ coincides with the set of all
   $C \in \BBb(\HHh,\KKk)$ such that $\Lambda(C) \subset \Lambda(A)$.
\end{enumerate}
\end{lem}
\begin{proof}
(a): Suppose $V \in \Upsilon(A)$. Let $P\dd \KKk \to A(V)$ be the orthogonal projection. Then we have
$P A_n\bigr|_V \to A\bigr|_V \in \GGg(V,A(V))$ and thus $P A_n\bigr|_V \in \GGg(V,A(V))$ as well for $n \geqsl N$.
This implies that $V \in \Upsilon(A_n)$,
\begin{equation}\label{eqn:aux3}
A_n(V) \in \Upsilon(P) \quad \textup{and} \quad P(A_n(V)) = A(V)
\end{equation}
and $\NnN(P) + A_n(V)$ is closed. So, $(\alpha,\beta,\gamma_n) := \xi_{A_n}(V) \in \Lambda(A_n)$ for $n \geqsl N$.
Now let $F_n$ denote the orthogonal complement (in $\KKk$) of $\NnN(P) + A_n(V)$. We see that
$\gamma_n = \dim(\KKk \ominus A_n(V)) = \dim (\NnN(P) \oplus F_n) = \dim (\NnN(P) \oplus P(F_n))
= \dim (\KKk \ominus P(A_n(V))) = \dim(\KKk \ominus A(V))$ which shows that
\begin{equation}\label{eqn:aux4}
\dim(\KKk \ominus A_n(V)) = \dim(\KKk \ominus A(V))\ (n \geqsl N)
\end{equation}
and $\xi_A(V) \in \Lambda(A_n)$. From now on, $n \geqsl N$. Let $P_n$ be the orthogonal projection of $\KKk$ onto $A_n(V)$.
Since $P_n A_n P_V = A_n P_V \to A P_V\ (n\to\infty)$ and simultaneously $\lim_{n\to\infty} P_n(A_n - A)P_V = 0$, we get
$\lim_{n\to\infty} P_n A P_V = A P_V$. Since $A\bigr|_V \in \GGg(V,A(V))$, the latter relation is equivalent to
\begin{equation}\label{eqn:aux6}
P_n P_{A(V)} \to P_{A(V)}\ (n \to \infty).
\end{equation}
Put $T_n = P_n\bigr|_{A(V)} \in \BBb(A(V),A_n(V))$. A straightforward calculation shows that
$T_n^* = P_{A(V)}\bigr|_{A_n(V)}$ (we compute $T_n^*$ as the adjoint of a member of $\BBb(A(V),A_n(V))$).
By \eqref{eqn:aux3}, $T_n^* \in \GGg(A_n(V),A(V))$ and hence $T_n \in \GGg(A(V),A_n(V))$. Let $T_n = Q_n |T_n|$
be the polar decomposition of $T_n$; that is, $Q_n \in \UUu(A(V),A_n(V))$ and $|T_n| \in \BBb_+(A(V))$.
Now by \eqref{eqn:aux4}, there is a unitary operator $Z_n$ on $\KKk$ which extends $Q_n$. It remains to prove that
\begin{equation}\label{eqn:z}
\lim_{n\to\infty} Z_n P_{A(V)} = P_{A(V)}.
\end{equation}
It follows from \eqref{eqn:aux6} that $\lim_{n\to\infty} P_{A(V)} P_n P_{A(V)} = P_{A(V)}$. We conclude from this that
$\lim_{n\to\infty} T_n^* T_n = I_{A(V)}$ and thus $\lim_{n\to\infty} |T_n| = I_{A(V)}$ and $|T_n|^{-1} \to I_{A(V)}\
(n\to\infty)$ as well. So, $Z_n P_{A(V)} = Q_n P_{A(V)} = T_n |T_n|^{-1} P_{A(V)} = (P_n P_{A(V)}) |T_n|^{-1} P_{A(V)}
\to P_{A(V)}\ (n\to\infty)$, which finishes the proof of (a).\par
The point (b) is immediate.
To prove (c), suppose $\Lambda(C) \subset \Lambda(A)$. For each $\epsi > 0$ we shall find $U \in \UUu(\KKk)$
and $G \in \GGg(\HHh)$ such that
\begin{equation}\label{eqn:aux0}
\|C - U A G^{-1}\| \leqsl \epsi.
\end{equation}
By \PRO{inside}, there is $V \in \Upsilon(C)$ such that
\begin{equation}\label{eqn:aux1}
\|C - C P_V\| \leqsl \frac12 \epsi.
\end{equation}
But then $\xi_C(V) \in \Lambda(C) \subset \Lambda(A)$, so we may find $W \in \Upsilon(A)$ such that $\xi_A(W) = \xi_C(V)$.
We conclude from this that there are unitary operators $U_0$ on $\HHh$ and $U$ on $\KKk$ with $U_0(W) = V$
and $U(A(W)) = C(V)$. For $n \geqsl 1$ define $G_n \in \GGg(\HHh)$ by $G_n = (C\bigr|_V)^{-1} U A$ on $W$
and $G_n = n U_0$ on $\HHh \ominus W$. Observe that $U A G_n^{-1} = C$ on $V$ and $\|U A G_n^{-1}\bigr|_{\HHh \ominus V}\|
\leqsl \frac1n \|A\|$. These properties combined with \eqref{eqn:aux1} yield \eqref{eqn:aux0} with $G = G_N$
for some $N \geqsl 1$. This shows that $C \in \cloRB_G^U(A)$ provided $\Lambda(C) \subset \Lambda(A)$. But the inclusion
$\cloRB_G^U(A) \subset \cloRB_G^G(A)$ is immediate and the implication `$C \in \cloRB_G^G(A) \implies \Lambda(C) \subset
\Lambda(A)$' follows from (a) and (b). So, $\cloRB_G^U(A) = \cloRB_G^G(A)$. Finally, if $C \in \cloRB_G^G(A)$, then
$C^* \in \cloRB_G^G(A^*)$ and therefore $C^* \in \cloRB_G^U(A^*)$ which yields $C \in \cloRB_U^G(A)$.
\end{proof}

Now for an operator $A \in \BBb(\HHh,\KKk)$ we define the following indices:
\begin{enumerate}[({I}1)]
\item \textit{range} index of $A$: $\iota_r(A) = \dim \overline{\RrR}(A)$ (cardinal),
\item \textit{initial} index of $A$: $\iota_i(A) = \dim \NnN(A) + \IC(\RrR(A))$ (cardinal),
\item \textit{final} index of $A$: $\iota_f(A) = \dim (\KKk \ominus \overline{\RrR}(A)) + \IC(\RrR(A))$ (cardinal),
\item \textit{binary} index of $A$: $\iota_b(A) = \bB(\RrR(A))$ ($0$ or $1$).
\end{enumerate}
For example, note that:
\begin{enumerate}[\upshape(P1)]
\item if $A$ is a closed range operator, then $\iota_i(A)$ and $\iota_f(A)$ are the well known indices: nullity and defect
   (respectively), compare e.g. \cite{corach},
\item $\iota_i(A^*) = \iota_f(A) = \iota_i(|A^*|) = \iota_f(|A^*|)$, $\iota_f(A^*) = \iota_i(A) = \iota_i(|A|)
   = \iota_f(|A|)$, $\iota_r(A^*) = \iota_r(A) = \iota_r(|A|) = \iota_r(|A^*|)$ and $\iota_b(A^*) = \iota_b(A)
   = \iota_b(|A|) = \iota_b(|A^*|)$ (see \PRO{adjoint}),
\item $\iota_i(A) + \iota_r(A) = \dim \HHh$, $\iota_f(A) + \iota_r(A) = \dim \KKk$ (to see this, consider separately
   the cases when $\iota_r(A)$ is finite-dimensional or not),
\item if $\KKk$ and $\HHh$ are separable and $\RrR(A)$ is nonclosed, $\iota_i(A) = \iota_f(A) = \aleph_0$.
\end{enumerate}
Using only these four indices we shall characterize all operators belonging to $\cloRB_G^G(A)$ (see \THM{orbit}).
To do this, we need

\begin{lem}{lambda}
For $A \in \BBb(\HHh,\KKk)$, $\Lambda(A)$ consists precisely of all the triples of the form
$(\iota_i(A)+\nu,\mu,\iota_f(A)+\nu)$ where $\mu$ and $\nu$ are cardinal numbers satisfying the conditions:
\begin{enumerate}[\upshape(a)]
\item $\mu + \nu = \iota_r(A)$,
\item if $\iota_b(A) = 0$, then $\mu < \iota_r(A)$.
\end{enumerate}
\end{lem}
\begin{proof}
Suppose $V \in \Upsilon(A)$ and let $(\alpha,\beta,\alpha') = \xi_A(V)$ and $\mu = \dim V$. Let $W$ be the orthogonal
complement (in $\HHh$) of the (closed) subspace $\NnN(A) + V$ and put $\nu = \dim W$. Then $A$ restricted to $V + W$
is a (continuous) monomorphism onto $\RrR(A)$ and $A(V)$ is closed in $\RrR(A)$. By \LEM{dim},
$\nu = \dim (\overline{\RrR}(A) \ominus A(V))$, which gives $\mu + \nu = \iota_r(A)$ and $\nu \geqsl \IC(\RrR(A))$.
Now \PRO{indices}--(a) yields that $\nu = \IC(\RrR(A)) + \nu$ and thus $\alpha = \iota_i(A) + \nu$
and $\alpha' = \iota_f(A) + \nu$. The condition (c) follows from the definition of $\iota_b(A)$ and the relation
$\mu = \dim A(V)$.\par
Now assume that $\mu$ and $\nu$ satisfy (a)--(c). If $\mu < \iota_r(A)$, then $\nu$ is uniquely determined by (a),
$\nu = \nu + \IC(\RrR(A))$ (because $\nu = \iota_r(A)$ if $\iota_r(A)$ is infinite and otherwise $\IC(\RrR(A)) = 0$)
and there is a complete subspace $E$ of $\RrR(A)$ with
\begin{equation}\label{eqn:mu}
\dim E = \mu
\end{equation}
(cf. \PRO{indices}--(d)). Then automatically
\begin{equation}\label{eqn:nu}
\nu + \IC(\RrR(A)) = \dim (\overline{\RrR}(A) \ominus E).
\end{equation}
If $\mu = \iota_r(A)$, then $\iota_b(A) = 1$ and (by \PRO{indices}--(c)) there is a complete subspace $W$ of $\RrR(A)$
such that $\dim W = \mu$ and $\dim (\overline{\RrR}(A) \ominus W) = \IC(\RrR(A))$. Thanks to (a) we may find a closed
subspace $F$ of $W$ such that $\dim (W \ominus F) = \mu$ and $\nu = \dim F$. Now putting $E = W \ominus F
\subset \RrR(A)$, we see that relations \eqref{eqn:mu} and \eqref{eqn:nu} are fulfilled.\par
To end the proof, let $V = A^{-1}(E) \cap \overline{\RrR}(A^*)$. Note that $V \in \Upsilon(A)$ and $A(V) = E$.
By \eqref{eqn:nu} and \LEM{dim}, $\nu + \IC(\RrR(A)) = \dim (\overline{\RrR}(A^*) \ominus V)$
and thus $(\iota_i(A) + \nu,\mu,\iota_f(A) + \nu) = \xi_A(V) \in \Lambda(A)$.
\end{proof}

The next result is simply deduced from the previous one.

\begin{cor}{lambda}
Let $A \in \BBb(\HHh,\KKk)$.
\begin{enumerate}[\upshape(i)]
\item If $\iota_r(A)$ is finite, then
   $$
   \Lambda(A) = \{(\dim \HHh - j,j,\dim \KKk - j)\dd\ j = 0,\ldots,\iota_r(A)\}
   $$
   where $\mM - j = \mM$ when $\mM$ is infinite and $j$ is finite.
\item If $\iota_b(A) = 0$, then
   $$
   \Lambda(A) = \{(\dim \HHh,\beta,\dim \KKk)\dd\ 0 \leqsl \beta < \iota_r(A)\}.
   $$
\item If $\iota_r(A)$ is infinite and $\iota_b(A) = 1$, then
   \begin{multline*}
   \Lambda(A) = \{(\dim \HHh,\beta,\dim \KKk)\dd\ 0 \leqsl \beta < \iota_r(A)\}\\
   \cup \{(\iota_i(A) + \alpha,\iota_r(A),\iota_f(A)+ \alpha)\dd\ 0 \leqsl \alpha \leqsl \iota_r(A)\}.
   \end{multline*}
\end{enumerate}
\end{cor}

Now the points (a) and (c) of \LEM{main} combined with \COR{lambda} yield

\begin{thm}{open}
Every operator $A \in \BBb(\HHh,\KKk)$ with $\iota_b(A) = 1$ has a neighbourhood $\XXx$ such that $A \in \cloRB_G^G(T)$
for all $T \in \XXx$.
\end{thm}
\begin{proof}
Notice that $(\iota_i(A),\iota_r(A),\iota_f(A)) \in \Lambda(A)$ and argue by contradiction: suppose there is a sequence
of bounded operators $A_1,A_2,\ldots$ which converge to $A$ and are such that $A \notin \cloRB_G^G(A_n)$ for each $n$.
But then \LEM{main}--(a) gives $(\iota_i(A),\iota_r(A),\iota_f(A)) \in \Lambda(A_n)$ for large $n$ and therefore
$\Lambda(A) \subset \Lambda(A_n)$ for this $n$ (by the formula for $\Lambda(X)$ given in \COR{lambda}). This denies
the point (c) of \LEM{main}.
\end{proof}

The main result of the paper is the following consequence of \COR{lambda} and \LEM{main}--(c). Its proof is omitted.

\begin{thm}{orbit}
For $A \in \BBb(\HHh,\KKk)$, $\cloRB_G^G(A)$ consists precisely of those operators $C \in \BBb(\HHh,\KKk)$ which satisfy
the following three conditions:
\begin{enumerate}[\upshape(a)]
\item $\iota_r(C) \leqsl \iota_r(A)$,
\item if $\iota_b(A) = 0$ and $\iota_b(C) = 1$, then $\iota_r(C) < \iota_r(A)$,
\item there is a cardinal $\alpha$ for which $\iota_i(C) = \iota_i(A) + \alpha$ and $\iota_f(C) = \iota_f(A) + \alpha$.
\end{enumerate}
What is more, if
\begin{equation}
\iota_i(C) = \dim \HHh \quad \textup{and} \quad \iota_f(C) = \dim \KKk,\tag{$\star$}
\end{equation}
then \textup{(c)} is fulfilled for any $A \in \BBb(\HHh,\KKk)$.
\end{thm}

\begin{rem}{Rank}
The description of $\cloRB_G^G(A)$ may be given in terms of only three indices: both the indices $\iota_r$ and $\iota_b$
may be `included' in one index $\iota_R$ defined by the rule: $\iota_R(X)$ is equal to $\iota_r(X)$ iff $\iota_b(X) = 0$,
otherwise $\iota_R(X)$ is the direct successor of $\iota_r(X)$ (in other words, $\iota_R(X)$ is the least cardinal $\alpha$
such that $\RrR(X)$ contains no complete subspace of dimension $\alpha$). Using \PRO{indices}--(e), one may show that
for $A, B \in \BBb(\HHh,\KKk)$ the points (a) and (b) of \THM{orbit} are fulfilled iff $\iota_R(B) \leqsl \iota_R(A)$.
So, $\cloRB_G^G(A)$ may be described by means of $\iota_R$, $\iota_i$ and $\iota_f$. However, it seems to us that the index
$\iota_R$ is rather unnatural, because $\iota_R(X)$ may be uncountable even if $X$ acts on a separable Hilbert space.\par
Moreover, using (P3), it may be shown that when $\dim \HHh \neq \dim \KKk$, the condition (c) of \THM{orbit} may be
simplified to `$\iota_m(C) \geqsl \iota_m(A)$' where $\iota_m(X) = \min(\iota_i(X),\iota_f(X))$
for $X \in \BBb(\HHh,\KKk)$. So, one needs only two indices ($\iota_R$ and $\iota_m$) for the description of $\cloRB_G^G$
provided $\HHh$ and $\KKk$ have different dimensions. The case when $\dim \HHh = \dim \KKk$ may simply be reduced
to the one when $\KKk = \HHh$. This case will be investigated in Section 5.
\end{rem}

\begin{cor}{the_same}
For $A, B \in \BBb(\HHh,\KKk)$, $\cloRB_G^G(A) = \cloRB_G^G(B)$ iff $\iota_r(A) = \iota_r(B)$, $\iota_i(A) = \iota_i(B)$,
$\iota_f(A) = \iota_f(B)$ and $\iota_b(A) = \iota_b(B)$.
\end{cor}

\begin{cor}{separable}
Let $A \in \BBb(\HHh)$.
\begin{enumerate}[\upshape(I)]
\item If $A$ is a finite rank operator, $\cloRB_G^G(A) = \{B \in \BBb(\HHh)\dd\ \iota_r(B) \leqsl \iota_r(A)\}$.
\item If $A$ is compact and of infinite rank, $\cloRB_G^G(A)$ coincides with the class of all compact operators
   on $\HHh$.
\item Suppose $\HHh$ is infinite-dimensional and separable.
   \begin{enumerate}[\upshape(i)]
   \item If $A$ is noncompact and nonsemi-Fredholm, $\cloRB_G^G(A)$ coincides with the class of all nonsemi-Fredholm
      operators on $\HHh$,
   \item If $A$ is semi-Fredholm, $\cloRB_G^G(A)$ consists of all nonsemi-Fred\-holm operators on $\HHh$ and of precisely
      those semi-Fred\-holm operators $B \in \BBb(\HHh)$ for which $\ind(B) = \ind(A)$ and $$\min(\iota_i(B),\iota_f(B))
      \geqsl \min(\iota_i(A),\iota_f(A)).$$
   \end{enumerate}
\end{enumerate}
\end{cor}

From Corollary \ref{cor:separable} and \THM{open} one may conclude the classical theorem that in a separable
Hilbert space all semi-Fredholm operators of the same (arbitrarily fixed) index form a connected open set (which in fact
is the interior of the closure of the orbit $\OOo_G^G$ of a one semi-Fredholm operator which is a monomorphism
or an epimorphism). It may also be easily infered that all of these open sets have the same boundary, which was first
shown by Mbekhta \cite{mbekhta}. For details and generalization see Section~5.\par
Since $\OOo_G^G(I_{\HHh}) = \GGg(\HHh)$, we obtain the following generalization of the result of Izumino and Kato
\cite{i-k}.

\begin{cor}{invert}
The norm closure of the group of all invertible operators on a Hilbert space $\HHh$ is the set of all $A \in \BBb(\HHh)$
such that $\iota_i(A) = \iota_f(A)$.
\end{cor}

Our last purpose of this section is to describe $\cloRB_U^U(A)$. We shall do this with use of the closures
$\overline{\OOo}\UUu(A^*A)$ and $\overline{\OOo}\UUu(AA^*)$ of the orbits $\OOo\UUu(A^*A)$ and $\OOo\UUu(AA^*)$
where $\OOo\UUu(X) = \{U X U^{-1}\dd\ U \in \UUu(\HHh)\}$ for $X \in \BBb(\HHh)$. The characterization
of the members of $\overline{\OOo}\UUu(X)$ will be made in a subsequent paper.

\begin{pro}{UU}
Let $A \in \BBb(\HHh,\KKk)$.
\begin{enumerate}[\upshape(a)]
\item If $\iota_i(A) \leqsl \iota_f(A)$, then $$\cloRB_U^U(A) = \{B \in \BBb(\HHh,\KKk)\dd\
   B^*B \in \overline{\OOo}\UUu(A^*A) \textup{ and } \iota_f(B) = \iota_f(A)\}.$$
\item If $\iota_i(A) \geqsl \iota_f(A)$, then $$\cloRB_U^U(A) = \{B \in \BBb(\HHh,\KKk)\dd\
   BB^* \in \overline{\OOo}\UUu(AA^*) \textup{ and } \iota_i(B) = \iota_i(A)\}.$$
\end{enumerate}
\end{pro}
\begin{proof}
(a): If $B = \lim_{n\to\infty} V_n A U_n^{-1}$ with unitary $U_n$'s and $V_n$'s, then $B^*B = \lim_{n\to\infty}
U_n A^*A U_n^{-1}$ and $A \in \cloRB_U^U(B)$ which gives $\cloRB_G^G(B) = \cloRB_G^G(A)$. Thus we infer from \COR{the_same}
that
\begin{equation}\label{eqn:aux8}
\iota_f(B) = \iota_f(A).
\end{equation}
Conversely, suppose \eqref{eqn:aux8} holds true and $B^*B = \lim_{n\to\infty} U_n A^*A U_n^{-1}$ for some
$U_n \in \UUu(\HHh)$. Then
\begin{equation}\label{eqn:aux9}
U_n |A| U_n^{-1} \to |B|\ (n \to \infty)
\end{equation}
as well. This implies that $\cloRB_G^G(|B|) = \cloRB_G^G(|A|)$ and hence $\iota_f(|B|) = \iota_f(|A|)$
(cf. \COR{the_same}). The latter connection combined with (P2) gives
\begin{equation}\label{eqn:aux10}
\iota_i(B) = \iota_i(A).
\end{equation}
Fix $\epsi > 0$ and take, using \PRO{inside}, $W \in \Upsilon(B)$ such that $|B|(W) = W$ and
\begin{equation}\label{eqn:aux11}
\|B - B P_W\| < \frac{\epsi}{4}.
\end{equation}
Since $W \in \Upsilon(|B|)$, from \eqref{eqn:aux9} and \LEM{main}--(a) it follows that $W \in \Upsilon(|A| U_n^{-1})
= \Upsilon(A U_n^{-1})$ for all but finitely many $n$'s. Passing to a subsequence, we may assume that this is true for
all $n$'s. Put $W_n = U_n |A| U_n^{-1}(W)$. Again by \LEM{main}--(a), there is a sequence of unitary operators
$Z_1,Z_2,\ldots$ on $\HHh$ such that $Z_n(W) = W_n$ and $\lim_{n\to\infty} Z_n P_W = P_W$. We infer from this that
$\lim_{n\to\infty} Z_n |B| P_W = |B| P_W$ and therefore
\begin{equation}\label{eqn:aux12}
Z_n^{-1} U_n |A| U_n^{-1} P_W \to |B| P_W\ (n \to \infty)
\end{equation}
(because $\|Z_n^{-1} U_n |A| U_n^{-1} P_W - |B| P_W\| = \|U_n |A| U_n^{-1} P_W - Z_n |B| P_W\|$).\par
Further, let $B = Q |B|$ be the polar decomposition of $B$. Since $\iota_f(A) \geqsl \iota_i(A)$, there is a cardinal
$\alpha$ for which $\iota_f(A) = \iota_i(A) + \alpha$. We claim that there is an isometry $V \in \BBb(\HHh,\KKk)$
such that $V\bigr|_W = Q\bigr|_W$ and $\dim(\KKk \ominus \RrR(V)) = \alpha$. Indeed, since $W$ is a complete subspace
of $\RrR(|B|)$, $\dim(\overline{\RrR}(|B|) \ominus W) \geqsl \IC(\RrR(|B|)) = \IC(\RrR(B))$ and hence there are orthogonal
closed subspaces $E$ and $F$ of $\HHh$ such that $\overline{\RrR}(|B|) \ominus W = E \oplus F$ and $\dim E = \IC(\RrR(B))$.
Then $\KKk \ominus Q(W) = [(\KKk \ominus \overline{\RrR}(B)) \oplus Q(E)] \oplus Q(F)$
and $\dim [(\KKk \ominus \overline{\RrR}(B)) \oplus Q(E)] = \dim (\KKk \ominus \overline{\RrR}(B)) + \IC(\RrR(B))
= \iota_f(B) = \iota_i(B) + \alpha$, thanks to \eqref{eqn:aux8} and \eqref{eqn:aux10}. Similarly, $\HHh \ominus W
= (\NnN(B) \oplus E) \oplus F$ and $\dim (\NnN(B) \oplus E) = \iota_i(B)$. This means that we may find suitable $V$
in such a way that it extends $Q\bigr|_{W \oplus F}$.\par
Now observe that $\iota_f(V Z_n^{-1} U_n |A| U_n^{-1}) = \alpha + \iota_f(|A|) = \iota_f(A) = \iota_f(A U_n^{-1})$
and $(V Z_n^{-1} U_n |A| U_n^{-1})^* (V Z_n^{-1} U_n |A| U_n^{-1}) = (A U_n^{-1})^* (A U_n^{-1})$. From \THM{U}
(see Section 6) it follows that $V Z_n^{-1} U_n |A| U_n^{-1} \in \cloRB^U(A U_n^{-1})$ and thus there is
$V_n \in \UUu(\KKk)$ for which
\begin{equation}\label{eqn:aux13}
\|V_n A U_n^{-1} - V Z_n^{-1} U_n |A| U_n^{-1}\| < \frac{\epsi}{4}.
\end{equation}
To this end, note that $\|V_n A U_n^{-1} (I_{\HHh} - P_W)\| = \|U_n |A| U_n^{-1} (I_{\HHh} - P_W)\| \to \|\,|B|(I_{\HHh}
- P_W)\| < \frac {\epsi}{4}$ (by \eqref{eqn:aux9} and \eqref{eqn:aux11}) and therefore for large $n$'s one has
$\|(V_n A U_n^{-1} - B)(I_{\HHh} - P_W)\| \leqsl \frac{\epsi}{2}$. While on the other hand, first making use
of \eqref{eqn:aux13} and next of \eqref{eqn:aux12},
\begin{multline*}
\|(V_n A U_n^{-1} - B)P_W\| \leqsl \|(V_n A U_n^{-1} - V Z_n^{-1} U_n |A| U_n^{-1})P_W\|\\
+ \|V Z_n^{-1} U_n |A| U_n^{-1} P_W - Q P_W |B|\,\|\\
\leqsl \frac{\epsi}{4} + \|V Z_n^{-1} U_n |A| U_n^{-1} P_W - V P_W |B|\,\| \to \frac{\epsi}{4}\ (n \to \infty)
\end{multline*}
which clearly shows that for some large $n$ we have $\|V_n A U_n^{-1} - B\| \leqsl \epsi$.\par
To prove (b), pass to adjoints and apply (a) (using (P2)).
\end{proof}

\begin{exm}{UO}
Let $\HHh$ be separable infinite-dimensional, $A \in \BBb_+(\HHh)$ be a noninvertible operator with dense range
and $V \in \BBb(\HHh)$ be a nonunitary isometry. By (P4), $\iota_f(VA) = \iota_f(A) = \aleph_0$. Moreover, $(VA)^* (VA)
= A^2 = A^* A$ and thus, by \PRO{UU}, $VA \in \cloRB_U^U(A)$ which easily gives $V A^2 V^* = (VA) (VA)^*
\in \UUu \overline{\OOo}(A^2)$. On the other hand, $V A^2 V^* \notin \OOo\UUu(A^2)$ (since $\overline{\RrR}(V A^2 V^*)
\neq \HHh$). This implies that both the operators $B := V A^2 V^* + I_{\HHh}$ and $C := A^2 + I_{\HHh}$ are nonnegative,
invertible, non-unitarily equivalent, but $B \in \overline{\OOo}\UUu(C)$. The example shows that the orbit $\OOo\UUu(X)$
is not closed in general (even when $X$ is invertible and nonnegative) and that the description of its closure seems to be
much more difficult than in case of the orbits investigated in this paper.
\end{exm}

\SECT{Application: ideals of $\BBb(\HHh)$}

With use of \THM{orbit}, we may easily point out all closed two-sided ideals of $\BBb(\HHh)$ for nonseparable Hilbert
space $\HHh$. For each infinite cardinal $\alpha \leqsl \dim \HHh$ let $J_{\alpha}$ be the set of all operators
$A \in \BBb(\HHh)$ such that $\iota_r(A) < \alpha$ or $\iota_r(A) = \alpha$ and $\iota_b(A) = 0$ (notice that
$J_{\aleph_0}$ is consists precisely of all compact operators). Our aim is to show that $J_{\alpha}$'s are the only
nontrivial ideals in $\BBb(\HHh)$.

\begin{lem}{binary}
For $A \in \BBb(\HHh)$ \tfcae
\begin{enumerate}[\upshape(a)]
\item $\iota_b(A) = 0$,
\item $A$ is the limit of a sequence $(A_n)_{n=1}^{\infty} \in \BBb(\HHh)$ such that $\iota_r(A_n) < \iota_r(A)$
   for each $n$.
\end{enumerate}
\end{lem}
\begin{proof}
From \PRO{inside} we conclude that (a) implies (b). The inverse implication follows from \LEM{main}--(a) and \COR{lambda}.
\end{proof}

The next result is probably known.

\begin{thm}{ideals}
Let $\HHh$ be a nonseparable Hilbert space. For each infinite $\alpha \leqsl \dim \HHh$, $J_{\alpha}$ is a closed two-sided
ideal in $\BBb(\HHh)$, $J_{\alpha} \neq J_{\alpha'}$ if $\alpha \neq \alpha'$ and every nonzero proper closed two-sided
ideal in $\BBb(\HHh)$ coincides with some $J_{\alpha}$.
\end{thm}
\begin{proof}
That $J_{\alpha}$ is a closed two-sided ideal it may easily be infered from \LEM{binary}. It is also immediate that
$J_{\alpha}$ uniquely determines $\alpha$. Let $J$ be a nonzero proper closed two-sided ideal of $\BBb(\HHh)$. Observe that
$J \subset J_{\dim \HHh}$. (Indeed, otherwise there would be $A \in J$ and $V \in \Upsilon(A)$ such that $\dim V
= \dim \HHh$. But then $I_{\HHh}$ would belong to $J$ since $I = X A Y$ for suitable $X, Y \in \BBb(\HHh)$.)
Let $\alpha \geqsl \aleph_0$ be the least cardinal for which $J \subset J_{\alpha}$. We shall show that $J = J_{\alpha}$.
To do this, thanks to \LEM{binary} it is enough to prove that $C \in J$ provided $\iota_r(C) < \alpha$.
A standard argument shows that a nonzero ideal contains all finite rank operators. So, we may assume that $\iota_r(C)$
is infinite. Since $J \not\subset J_{\beta}$ with $\beta = \iota_r(C)$, there is $A \in J$ such that
$\iota_r(A) > \iota_r(C)$ or $\iota_r(A) = \iota_r(C)$ and $\iota_b(A) = 1$. But then, by \THM{orbit},
$C \in \cloRB_G^G(A) \subset J$ (note that ($\star$) is fulfilled for $C$ since $\iota_r(C) < \dim \HHh$).
\end{proof}

\SECT{Indices $\ind$ and $\iota_m$}

In this section $\HHh$ denotes an infinite-dimensional Hilbert space. Our aim is to define $\ind(A)$ for certain operators
$A \in \BBb(\HHh)$ in such a way that this new index extends the well known one (denoted in the same way) for semi-Fredholm
operators.\par
Let $\Dd$ be the class of all pairs $(\alpha,\beta)$ of cardinals such that either $\alpha = \beta < \aleph_0$ or $\alpha$
and $\beta$ are different. For $(\alpha,\beta) \in \Dd$ we define $\alpha - \beta$ as a cardinal or the negative
of a cardinal in a very natural way:
\begin{itemize}
\item if $\alpha$ and $\beta$ are finite, $\alpha - \beta$ is the difference of $\alpha$ and $\beta$ treated as natural
   numbers,
\item if $\alpha > \beta$ and $\alpha$ is infinite, $\alpha - \beta := \alpha$,
\item if $\alpha < \beta$ and $\beta$ is infinite, $\alpha - \beta := -\beta$.
\end{itemize}
Additionally, for simplicity, let us agree with the following notation: $|\alpha| = |-\alpha| := \alpha$ for every cardinal
$\alpha$.\par
Now for any $A \in \BBb(\HHh)$ such that $(\iota_i(A),\iota_f(A)) \in \Dd$ let $\ind(A) = \iota_i(A) - \iota_f(A)$.
Notice that when $\iota_i(A) = \iota_f(A) \geqsl \aleph_0$, $\ind(A)$ is undefined.\par
For every $\gamma$ such that $|\gamma| \leqsl \dim \HHh$ denote by $\Ind_{\gamma}(\HHh)$ the set of all operators
$A \in \BBb(\HHh)$ for which $\ind(A)$ is defined and $\ind(A) = \gamma$, and let $\overline{\Ind}_{\gamma}(\HHh)$ be
the closure of $\Ind_{\gamma}(\HHh)$. Finally, let $\Uind(\HHh)$ stand for the set of all operators for which $\ind$
is undefined. Recall also (see \REM{Rank}) that $\iota_m(A) = \min(\iota_i(A),\iota_f(A))$.\par
We leave this as a simple exercise that $\iota_b(A) = 1$ for $A \in \BBb(\HHh) \setminus \Uind(\HHh)$.\par
The following is a reformulation of \THM{orbit}:

\begin{cor}{orbit}
Let $A \in \BBb(\HHh)$.
\begin{enumerate}[\upshape(I)]
\item If $A \in \Uind(\HHh)$, $\cloRB_G^G(A)$ is the set of all $B \in \Uind(\HHh)$ such that:
   \begin{enumerate}[\upshape(i)]
   \item $\iota_r(B) \leqsl \iota_r(A)$; and $\iota_r(B) < \iota_r(A)$ provided $\iota_b(A) = 0$ and $\iota_b(B) = 1$,
   \item $\iota_m(B) \geqsl \iota_m(A)$.
   \end{enumerate}
\item If $A \notin \Uind(\HHh)$ and $\gamma := \ind(A)$, $\cloRB_G^G(A)$ is the set of all $B \in \Ind_{\gamma}(\HHh)$
   for which $\iota_r(B) \leqsl \iota_r(A)$ and $\iota_m(B) \geqsl \iota_m(A)$ and of all $C \in \Uind(\HHh)$ such that
   $\iota_r(C) \leqsl \iota_r(A)$ and $\iota_m(C) \geqsl |\gamma|$.
\end{enumerate}
\end{cor}

The main result of the section is

\begin{thm}{int-bd}
For every $\gamma$ with $|\gamma| \leqsl \dim \HHh$ the set $\Ind_{\gamma}(\HHh)$ is connected and open in $\BBb(\HHh)$
and it coincides with the interior of its closure. The boundary of $\Ind_{\gamma}(\HHh)$ is connected as well and consists
precisely of all $A \in \Uind(\HHh)$ such that $\iota_m(A) \geqsl |\gamma|$.
\end{thm}
\begin{proof}
Let us first show that
\begin{equation}\label{eqn:Ind}
\overline{\Ind}_{\gamma}(\HHh) = \Ind_{\gamma}(\HHh) \cup \{A \in \Uind(\HHh)\dd\ \iota_m(A) \geqsl |\gamma|\}.
\end{equation}
Let $Z$ be a closed range operator which is a monomorphism or an epimorphism and for which $\ind(Z) = \gamma$.
Notice that $\iota_r(Z) = \dim \HHh$, $\iota_m(Z) = 0$ and $\oRB_G^G(Z) \subset \Ind_{\gamma}(\HHh)$. What is more,
we infer from \COR{orbit} that $\cloRB_G^G(Z)$ coincides with the right hand side expression of \eqref{eqn:Ind}.
This shows that \eqref{eqn:Ind} holds true and that $\Ind_{\gamma}(\HHh)$ is connected (since $\oRB_G^G(Z)$ is connected,
by the connectedness of $\GGg(\HHh)$).\par
Further, by \eqref{eqn:Ind}, $\overline{\Ind}_{\gamma}(\HHh) \cap \overline{\Ind}_k(\HHh) = \{B \in \Uind(\HHh)\dd\
\iota_m(B) \geqsl |\gamma|\}$ for each integer $k \neq \gamma$ and thus the latter set is contained in the boundary
of $\overline{\Ind}_{\gamma}(\HHh)$. So, to end the proof, it is enough to show that $\Ind_{\gamma}(\HHh)$ is open.\par
Fix $A \in \Ind_{\gamma}(\HHh)$. Since $\iota_b(A) = 1$, by \THM{open} there is a neighbourhood $\XXx$ of $A$ such that
\begin{equation}\label{eqn:aux7}
A \in \cloRB_G^G(X)
\end{equation}
for any $X \in \XXx$. Note that $\Uind(\HHh)$ is closed (by \eqref{eqn:Ind}: $\Uind(\HHh) = \overline{\Ind}_0(\HHh)
\cap \overline{\Ind}_1(\HHh)$) and therefore we may assume that $\XXx$ is disjoint from $\Uind(\HHh)$. But then, thanks
to \eqref{eqn:aux7} and \COR{orbit}, $\ind(A) = \ind(X)$ for $X \in \XXx$ and hence $\XXx \subset \Ind_{\gamma}(\HHh)$.\par
Finally, to show that the boundary of $\Ind_{\gamma}(\HHh)$ is connected, take a closed range operator $T \in \Uind(\HHh)$
such that $\iota_r(T) = \dim \HHh$ and $\iota_m(T) = \max(\aleph_0,|\gamma|)$ and observe, applying again \COR{orbit},
that the boundary coincides with $\cloRB_G^G(T)$.
\end{proof}

The above result shows that $\Uind(\HHh)$ is closed, nowhere dense and connected. Notice also that $\Ind_{\gamma}(\HHh)$
consists of closed range operators iff $\gamma \in \ZZZ \cup \{-\aleph_0,\aleph_0\}$. So, in case of a nonseparable Hilbert
space $\HHh$ we may define a \textit{semi-Fredholm} operator on $\HHh$ as a bounded operator $A \notin \Uind(\HHh)$
such that $\ind(A) \in \ZZZ \cup \{-\aleph_0,\aleph_0\}$. Under such a definition, semi-Fredholm operators automatically
have closed ranges. Observe also that in a separable Hilbert space the class $\Uind$ coincides with the class of all
non-semi-Fredholm operators. So, \THM{int-bd} generalizes the result of Mbekhta \cite{mbekhta}.\par
With use of \COR{orbit} we are also able to show

\begin{pro}{cut}
For each $\gamma$ with $|\gamma| \leqsl \dim \HHh$ and a positive cardinal $\mM$ the set $\Ind_{\gamma}^{\mM}(\HHh)$ of all
$A \in \Ind_{\gamma}(\HHh)$ for which $\iota_m(A) < \mM$ is open and dense in $\Ind_{\gamma}(\HHh)$.
\end{pro}
\begin{proof}
First of all note that if $\mM \geqsl \max(\aleph_0,|\gamma|)$, then $\Ind_{\gamma}^{\mM}(\HHh) = \Ind_{\gamma}(\HHh)$.
So, we may assume that
\begin{equation}\label{eqn:aux5}
\mM < \max(\aleph_0,|\gamma|).
\end{equation}
Let $Z$ be as in the proof of \THM{int-bd}. Observe that $\oRB_G^G(Z) \subset \Ind_{\gamma}^{\mM}(\HHh)$ and therefore
the latter set is dense in $\Ind_{\gamma}(\HHh)$. What is more, thanks to \eqref{eqn:aux5}, there exists a closed range
operator $Z_{\mM} \in \BBb(\HHh) \cap \Ind_{\gamma}(\HHh)$ such that $\iota_m(Z_{\mM}) = \mM$ and $\iota_r(Z_{\mM})
= \dim \HHh$. Now by \COR{orbit}, $\Ind_{\gamma}^{\mM}(\HHh) = \Ind_{\gamma}(\HHh) \setminus \cloRB_G^G(Z_{\mM})$
which finishes the proof.
\end{proof}

\begin{cor}{nowhere_dense}
The closure of $\oRB_G^G(A)$ has nonempty interior iff $A$ or $A^*$ is an epimorphism.
\end{cor}
\begin{proof}
The sufficiency is clear ($\oRB_G^G(A)$ is open provided $A$ or $A^*$ is an epimorphism). To see the necessity, first
note that the nonemptiness of the interior of $\cloRB_G^G(A)$ implies that $A \notin \Uind(\HHh)$, and it suffices
to show that $\mM := \iota_m(A) = 0$ (the latter condition is equivalent to the epimorphicity of $A$ or $A^*$). Suppose,
for the contrary that $\mM > 0$. Then $\oRB_G^G(A) \subset \overline{\Ind}_{\gamma}(\HHh) \setminus
\Ind_{\gamma}^{\mM}(\HHh)$ where $\gamma = \ind(A)$. Now it follows from \PRO{cut} that $\OOo_G^G(A)$ is nowhere dense.
A contradiction.
\end{proof}

\begin{rem}{ind}
The index $\ind$ may clearly be defined by the same formula in spaces $\BBb(\HHh,\KKk)$. All the results of the section
have their (natural) counterparts in such spaces when $\dim \HHh = \dim \KKk \geqsl \aleph_0$, that is,
when $(\dim \HHh,\dim \KKk) \notin \Dd$. In the opposite, when $(\dim \HHh,\dim \KKk) \in \Dd$, one may easily prove
(using (P3)) that $(\iota_i(X),\iota_f(X)) \in \Dd$ and $\ind(X) = \dim \HHh - \dim \KKk$
for every $X \in \BBb(\HHh,\KKk)$. So, the restriction (in this section) of our investigations to operators acting
on a one Hilbert space was reasonable and justified.
\end{rem}

\SECT{One-sided actions}

\begin{thm}{U}
Let $A \in \BBb(\HHh,\KKk)$.
\begin{enumerate}[\upshape(a)]
\item $\cloRB^U(A)$ is the set of all $B \in \BBb(\HHh,\KKk)$ such that
   \begin{equation}\label{eqn:U-}
   B^* B = A^* A \quad \textup{and} \quad \iota_f(B) = \iota_f(A).
   \end{equation}
\item $\cloRB_U(A) = \{B \in \BBb(\HHh,\KKk)\dd\ B B^* = A A^* \textup{ and } \iota_i(B) = \iota_i(A)\}$.
\end{enumerate}
\end{thm}
\begin{proof}
(a): First of all, note that $B \in \cloRB^U(A)$ iff $A \in \cloRB^U(B)$. So, thanks to \COR{the_same}, $\iota_f(B)
= \iota_f(A)$ for $B \in \cloRB^U(A)$. It is also clear that $B^* B = A^* A$ for such $B$.\par
Conversely, take $B$ satisfying \eqref{eqn:U-} and put $D = |A|\ (= |B|)$. Fix $\epsi > 0$ and take a subspace
$V \in \Upsilon(D)$ contained in $\RrR(D)$ such that
\begin{equation}\label{eqn:aux2}
\|D - D P_V\| \leqsl \epsi
\end{equation}
(cf. \PRO{inside}). Since $D(V)$ is a complete subspace of $\RrR(D)$, $\dim (\overline{\RrR}(D) \ominus D(V))
\geqsl \IC(\RrR(D))$ and thus there are mutually orthogonal closed subspaces $E$ and $F$ of $\HHh$ such that
$\overline{\RrR}(D) \ominus D(V) = E \oplus F$ and $\dim E = \IC(\RrR(D))$. Recall that $\IC(\RrR(A)) = \IC(\RrR(D))
= \IC(\RrR(B))$. Let $A = Q_A D$ and $B = Q_B D$ be the polar decompositions of $A$ and $B$, respectively. Observe that
$Q_A(D(V) \oplus E \oplus F) = A(V) \oplus Q_A(E) \oplus Q_A(F)$ and hence $\dim (\KKk \ominus A(V)) = \dim F
+ \iota_f(A)$. For the same reason, $\dim (\KKk \ominus B(V)) = \dim F + \iota_f(B)$. So, by \eqref{eqn:U-}, there is
a unitary operator $U_0$ of $\KKk \ominus A(V)$ onto $\KKk \ominus B(V)$. Now it suffices to define $U \in \UUu(\KKk)$ by:
$U = Q_B (Q_A\bigr|_{D(V)})^{-1}$ on $A(V)$ and $U = U_0$ on the orthogonal complement of $A(V)$. Finally we have
$UA\bigr|_V = B|_V$ and therefore $\|U A - B\| \leqsl 2\epsi$ (by \eqref{eqn:aux2}).\par
In order to prove (b), pass to adjoint operators and apply (a).
\end{proof}

\begin{cor}{separ}
Let $\HHh$ be a separable Hilbert space and let $A \in \BBb(\HHh)$ be such that $\RrR(A)$ is nonclosed.
Then $\cloRB^U(A) = \{B \in \BBb(\HHh)\dd\ B^* B = A^* A\}$ and $\cloRB_U(A) = \{B \in \BBb(\HHh)\dd\ B B^* = A A^*\}$.
\end{cor}

The case of the closures of orbits $\OOo_G$ and $\OOo^G$ is much more complicated. For need of their descriptions,
let us define $L_+(A)$ for $A \in \BBb_+(\HHh)$ as the set of all $B \in \BBb_+(\HHh)$ such that $B \leqsl c A$ for some
scalar $c > 0$, and let $\LLl_+(A)$ be the closure of $L_+(A)$. By Theorem~2.1 of \cite{ranges}, for an operator
$B \in \BBb_+(\HHh)$,
\begin{equation}\label{eqn:range}
B \in L_+(A) \iff \RrR(\sqrt{B}) \subset \RrR(\sqrt{A}),
\end{equation}
iff $\sqrt{B} = \sqrt{A} T$ for some $T \in \BBb(\HHh)$. Observe that the latter condition gives
$B = \sqrt{A} TT^* \sqrt{A}$. Conversely, if $B = \sqrt{A} C \sqrt{A}$ for some $C \in \BBb_+(\HHh)$,
then $B \leqsl \|C\| A$, that is, $B \in L_+(A)$. Thus we have obtained that, whenever $A, B \in \BBb_+(\HHh)$:
\begin{equation}\label{eqn:L+}
B \in L_+(A) \iff \exists C \in \BBb_+(\HHh)\dd\ B = \sqrt{A} C \sqrt{A}.
\end{equation}
It is clear that $\LLl_+(A)$ is a cone (i.e. $tB + sC \in \LLl_+(A)$ whenever $B, C \in \LLl_+(A)$ and $t,s \geqsl 0$).
Other properties of $\LLl_+(A)$ are established in the following

\begin{pro}{L+}
Let $A \in \BBb_+(\HHh)$.
\begin{enumerate}[\upshape(a)]
\item $B \in \LLl_+(A) \implies \overline{\RrR}(B) \subset \overline{\RrR}(A)$.
\item $B \in \LLl_+(A) \implies \LLl_+(B) \subset \LLl_+(A)$.
\item Let $B \in \BBb_+(\HHh)$ and $E\dd \Bb(\RRR_+) \to \BBb(\HHh)$ be the spectral measure of $B$.
   Then $B \in \LLl_+(A)$ iff $E([\frac1n,+\infty)) \in \LLl_+(A)$ for each $n \geqsl 1$.
\item Let $P$ be an orthogonal projection. $P \in \LLl_+(A)$ iff there is a sequence $P_1,P_2,\ldots$ of orthogonal
   projections which converge to $P$ and whose ranges are contained in the range of $\sqrt{A}$.
\item If $B \in \BBb_+(\HHh)$ is compact and $\RrR(B) \subset \overline{\RrR}(A)$, then $B \in \LLl_+(A)$.
\end{enumerate}
\end{pro}
\begin{proof}
The point (a) follows from \eqref{eqn:L+} and the connection $\overline{\RrR}(A) = \overline{\RrR}(\sqrt{A})$.\par
(b): Suppose $B = \lim_{n\to\infty} B_n$ with $B_n \in L_+(A)$. Then $\sqrt{B} = \lim_{n\to\infty} \sqrt{B_n}$.
Now if $C \in L_+(B)$, $C = \sqrt{B} D \sqrt{B}$ for some $D \in \BBb_+(\HHh)$ (by \eqref{eqn:L+}). So,
$C = \lim_{n\to\infty} \sqrt{B_n} D \sqrt{B_n}$. But (again by \eqref{eqn:L+}) $\sqrt{B_n} D \sqrt{B_n} \in L_+(B_n)
\subset L_+(A)$. This shows that $L_+(B) \subset \LLl_+(A)$ and we are done.\par
(c): Let $P_n = E([\frac1n,+\infty))$. Note that $BP_n \leqsl B$, $BP_n \to B\ (n \to \infty)$
and $\frac1n P_n \leqsl BP_n \leqsl \|B\|P_n$. So, it suffices to apply (b).\par
(d): The sufficiency follows from \eqref{eqn:range}. To prove the necessity, take a sequence $A_1,A_2,\ldots \in \LLl_+(A)$
convergent to $P$. Let $V = \RrR(P) \in \Upsilon(P)$ and let $N \geqsl 1$ and $Z_N,Z_{N+1},\ldots$ be as in \LEM{main}--(a)
for $\KKk := \HHh$ and $A := P$. We may assume that $N = 1$. Put $P_n = P_{A_n(V)}$. Observe that $\RrR(P_n)
\subset \RrR(\sqrt{A})$ (since $A_n \in L_+(A)$ and thanks to \eqref{eqn:range}). Finally, $P_n = Z_n P Z_n^{-1}$
(because $Z_n(V) = A_n(V)$) and therefore $\lim_{n\to\infty} P_n = \lim_{n\to\infty}(Z_n P) (Z_n P)^* = P \cdot P^*
= P$.\par
(e): Thanks to (c), it suffices to show that every finite rank orthogonal projection whose image is contained
in $\overline{\RrR}(A)$ is a member of $\LLl_+(A)$ which we leave as a simple exercise.
\end{proof}

\begin{cor}{comp}
For a compact operator $A \in \BBb_+(\HHh)$, $\LLl_+(A)$ consists of all compact operators $B \in \BBb_+(\HHh)$
such that $\RrR(B) \subset \overline{\RrR}(A)$.
\end{cor}

\begin{exm}{L+}
Let $(\HHh,\scalarr)$ be infinite-dimensional and separable, and let $(e_n)_{n=1}^{\infty}$ be an orthonormal basis
of $\HHh$. Put $A\dd \HHh \oplus \HHh \ni (x,y) \mapsto (x,\sum_{n=1}^{\infty} \frac{\scalar{y}{e_n}}{2^n} e_n)
\in \HHh \oplus \HHh$, $V = \{(x,y) \in \HHh \oplus \HHh\dd\ x = 0\}$ and $U\dd \HHh \oplus \HHh \ni (x,y) \mapsto (y,x)
\in \HHh \oplus \HHh$. Observe that $A \in \BBb_+(\HHh \oplus \HHh)$, $V \subset \bar{\RrR}(A) = \HHh \oplus \HHh$,
$U \in \UUu(\HHh \oplus \HHh)$ and $U(V) \subset \RrR(A)$. However, $P_V \notin \LLl_+(A)$. The example shows that
if $A$ is noncompact and the range of $A$ is nonclosed, the description of $\LLl_+(A)$ is not so easy as stated
in \COR{comp}. This issue will be investigated elsewhere.
\end{exm}

As the next result shows, the cones $\LLl_+(AA^*)$ and $\LLl_+(A^*A)$ play an important role in the description
of $\cloRB_G(A)$ and $\cloRB^G(A)$.

\begin{thm}{G}
Let $A \in \BBb(\HHh,\KKk)$.
\begin{enumerate}[\upshape(a)]
\item $\cloRB_G(A)$ consists of precisely those $B \in \BBb(\HHh,\KKk)$ such that $BB^* \in \LLl_+(AA^*)$ and
   \begin{equation}\label{eqn:one}
   \iota_i(B) = \iota_i(A) + \iota_i(B^*\bigr|_{\overline{\RrR}(A)}).
   \end{equation}
\item $\cloRB^G(A)$ consists of precisely those $B \in \BBb(\HHh,\KKk)$ such that $B^*B \in \LLl_+(A^*A)$
   and $\iota_f(B) = \iota_f(A) + \iota_i(B\bigr|_{\overline{\RrR}(A^*)})$.
\end{enumerate}
\end{thm}
\begin{proof}
Since (b) may be infered from (a) by passing to adjoints, we only need to show (a).\par
First suppose that $B \in \cloRB_G(A)$. This means that $B = \lim_{n\to\infty} A G_n$ for some $G_n \in \GGg(\HHh)$
and thus $BB^* = \lim_{n\to\infty} |A^*| Q G_n G_n^* Q^* |A^*|$ where $Q$ is the partial isometry appearing in the polar
decomposition of $A$. So, it follows from \eqref{eqn:L+} that $BB^* \in \LLl_+(AA^*)$. What is more, the latter implies
that $\overline{\RrR}(B) \subset \overline{\RrR}(A)$. For simplicity, put $\KKk_0 = \overline{\RrR}(A)$ and think
of $A$ and $B$ as members of $\BBb(\HHh,\KKk_0)$. Under such a consideration, $B \in \cloRB_G^G(A)$ and hence \THM{orbit}
implies that $\iota_i(B) = \iota_i(A) + \alpha$ and $\iota_f(B) = \iota_f(A) + \alpha$ for some cardinal $\alpha$
where all the indices which appear in both the equations are computed in the space $\BBb(\HHh,\KKk_0)$. Notice that
then $\iota_f(A) = \IC(\RRr(A))$, so $\iota_f(B) = \IC(\RrR(A)) + \alpha$ and (by \PRO{indices}--(a))
$\iota_i(B) = \iota_i(A) + \IC(\RrR(A)) + \alpha = \iota_i(A) + \iota_f(B)$. It suffices to observe that (still
in the space $\BBb(\HHh,\KKk_0)$) $\iota_f(B) = \dim (\overline{\RrR}(A) \ominus \overline{\RrR}(B)) + \IC(\RrR(B))
= \iota_i(B^*\bigr|_{\overline{\RrR}(A)})$ which finally gives \eqref{eqn:one}.\par
Now suppose that $BB^* \in \LLl_+(AA^*)$ and \eqref{eqn:one} is satisfied. As in the first part of the proof, notice
that then $\overline{\RrR}(B) \subset \overline{\RrR}(A)$ and thus $\iota_i(B^*\bigr|_{\overline{\RrR}(A)})
= \dim(\overline{\RrR}(A) \ominus \overline{\RrR}(B)) + \IC(\RrR(B))$. So, \eqref{eqn:one} is equivalent to
\begin{equation}\label{eqn:sec}
\iota_i(B) = \iota_i(A) + \IC(\RrR(B)) + \dim(\overline{\RrR}(A) \ominus \overline{\RrR}(B)).
\end{equation}
Fix $\epsi > 0$
and take $V \in \Upsilon(|B^*|)$ such that $|B^*|(V) = V$ and $\|\,|B^*| - |B^*|P_V\| \leqsl \epsi$. Since $V$
is a complete subspace of $\RrR(|B^*|) = \RrR(B)$, we see that $\dim(\overline{\RrR}(B) \ominus V) \geqsl \IC(\RrR(B))$
and thus
\begin{equation}\label{eqn:ic}
\dim(\overline{\RrR}(B) \ominus V) = \dim(\overline{\RrR}(B) \ominus V) + \IC(\RrR(B)).
\end{equation}
Let $B = Q |B|$
be the polar decompositions of $B$. Then also $B = |B^*| Q$. Put $W = Q^{-1}(V) \cap \overline{\RrR}(B^*)$. Note that
$W \in \Upsilon(B)$, $Q P_W = P_V Q$,
\begin{equation}\label{eqn:W}
B(W) = V \qquad \textup{and} \qquad \|B - B P_W\| < \epsi.
\end{equation}
Further, since $BB^* \in \LLl_+(AA^*)$, by \eqref{eqn:L+}, there is a sequence $T_1,T_2,\ldots$ of bounded nonnegative
operators on $\KKk$ such that $$BB^* = \lim_{n\to\infty} |A^*| T_n |A^*|.$$ We conclude from the relations
$V \in \Upsilon(|B^*|)$ and $|B^*|(V) = V$ that $V \in \Upsilon(BB^*)$ as well and therefore, thanks to \LEM{main}--(a),
after omitting finitely many entries of $(|A^*| T_n |A^*|)_{n=1}^{\infty}$, there is a sequence $(Z_n)_{n=1}^{\infty}$
of unitary operators on $\KKk$ such that
\begin{equation}\label{eqn:Vn}
Z_n(V) = V_n \textup{ for each } n
\end{equation}
where $V_n := |A^*|T_n |A^*|(V)$ is a closed subspace of $\KKk$, and
\begin{equation}\label{eqn:PV}
Z_n P_V \to P_V\ (n \to \infty).
\end{equation}
What is more, since we might restrict our argument (when taking $Z_n$) to $\overline{\RrR}(A)$
(and work in $\BBb(\HHh,\overline{\RrR}(A))$ and $\UUu(\overline{\RrR}(A))$), we may also assume that
\begin{equation}\label{eqn:ZnR}
Z_n(\overline{\RrR}(A)) = \overline{\RrR}(A) \textup{ for every } n.
\end{equation}
Observe that $V_n \subset \RrR(|A^*|) = \RrR(A)$ and thus $W_n \in \Upsilon(A)$ where $W_n := A^{-1}(V_n) \cap
\overline{\RrR}(A^*)$ and
\begin{equation}\label{eqn:ic2}
\dim(\overline{\RrR}(A) \ominus V_n) = \dim(\overline{\RrR}(A) \ominus V_n) + \IC(\RrR(A))
\end{equation}
(compare the proof of \eqref{eqn:ic}). We have:
\begin{equation}\label{eqn:AWn}
A(W_n) = V_n
\end{equation}
and, by \eqref{eqn:Vn},
\begin{equation}\label{eqn:dim}
\dim W_n = \dim V_n = \dim Z_n(V) = \dim V = \dim W.
\end{equation}
Now \LEM{dim} (applied twice) combined with \eqref{eqn:AWn}, \eqref{eqn:ic2}, \eqref{eqn:ZnR}, \eqref{eqn:Vn},
\eqref{eqn:ic} (twice) and \eqref{eqn:sec} yields
\begin{multline*}
\dim(\HHh \ominus W_n) = \dim \NnN(A) + \dim(\overline{\RrR}(A^*) \ominus W_n)\\
= \dim \NnN(A) + \dim(\overline{\RrR}(A) \ominus V_n)\\
= \dim \NnN(A) + \IC(\RrR(A)) + \dim(Z_n(\overline{\RrR}(A)) \ominus Z_n(V))\\
= \iota_i(A) + \dim(\overline{\RrR}(A) \ominus V)\\
= \iota_i(A) + \dim(\overline{\RrR}(A) \ominus \overline{\RrR}(B)) + \dim(\overline{\RrR}(B) \ominus V)\\
= \iota_i(A) + \dim(\overline{\RrR}(A) \ominus \overline{\RrR}(B)) + \IC(\RrR(B)) + \dim(\overline{\RrR}(B) \ominus V)\\
= \iota_i(B) + \dim(\overline{\RrR}(B) \ominus V) = \dim \NnN(B) + \IC(\RrR(B)) + \dim(\overline{\RrR}(B) \ominus V)\\
= \dim \NnN(B) + \dim(\overline{\RrR}(B) \ominus V) = \dim \NnN(B) + \dim(\overline{\RrR}(B^*) \ominus W)\\
= \dim (\HHh \ominus W).
\end{multline*}
The above connection and \eqref{eqn:dim} imply that there is $U_n \in \UUu(\HHh)$ for which $U_n(W) = W_n$. Now define
$G_n \in \GGg(\HHh)$ by: $$G_n\bigr|_W = (A\bigr|_{W_n})^{-1} Z_n B\bigr|_W \in \GGg(W,W_n)$$ (use \eqref{eqn:W},
\eqref{eqn:Vn} and \eqref{eqn:AWn} to see that $G_n\bigr|_W$ is well defined) and $G_n\bigr|_{\HHh \ominus W} = \frac1n
U_n\bigr|_{\HHh \ominus W} \in \GGg(\HHh \ominus W,\HHh \ominus W_n)$. We claim that
\begin{equation}\label{eqn:fin}
A G_n \to B P_W\ (n \to \infty).
\end{equation}
Indeed, $\lim_{n\to\infty} A G_n (I_{\HHh} - P_W) = \lim_{n\to\infty} \frac1n A U_n\bigr|_{\HHh \ominus W} = 0
= B P_W (I_{\HHh} - P_W)$ and, thanks to \eqref{eqn:W} and \eqref{eqn:PV}, $$\lim_{n\to\infty} A G_n P_W =
\lim_{n\to\infty} Z_n B P_W = \lim_{n\to\infty} Z_n P_V B P_W = P_V B P_W = B P_W.$$
Finally, we infer from \eqref{eqn:W} and \eqref{eqn:fin} that $\|A G_n - B\| \leqsl \epsi$ for some $n$, which finishes
the proof.
\end{proof}

The next result has its natural counterpart for the closures of $\oRB^G$.

\begin{cor}{G}
Let $A, B \in \BBb(\HHh,\KKk)$.
\begin{enumerate}[\upshape(I)]
\item $\cloRB_G(A) = \cloRB_G(B)$ iff $\LLl_+(AA^*) = \LLl_+(BB^*)$ and $\iota_i(B) = \iota_i(A)$.
\item Suppose $A$ is compact.
   \begin{enumerate}[\upshape(a)]
   \item $\cloRB_G(A)$ constists of all compact operators $C \in \BBb(\HHh,\KKk)$ such that $\RrR(C) \subset
      \overline{\RrR}(A)$.
   \item $\cloRB_G(A) = \cloRB_G(B)$ iff $B$ is compact and $\overline{\RrR}(B) = \overline{\RrR}(A)$.
   \end{enumerate}
\end{enumerate}
\end{cor}
\begin{proof}
The point (I) follows from \THM{G} and the fact that $\overline{\RrR}(B) = \overline{\RrR}(A)$
and $\iota_i(B^*\bigr|_{\overline{\RrR}(B)}) = \IC(\RrR(B))$ provided $\LLl_+(BB^*) = \LLl_+(AA^*)$.\par
To see (II), it suffices to apply \COR{comp} after observing that when $A$ and $B$ are compact
and $\RrR(B) \subset \overline{\RrR}(A)$, then \eqref{eqn:one} is fulfilled (consider separately the cases
when $\iota_r(A)$ is finite; $\HHh$ is separable and $\RrR(A)$ is nonclosed; and $\HHh$ is nonseparable).
\end{proof}

\end{document}

%% file: pream.tex
\usepackage{latexsym}
\usepackage{ifthen}
\usepackage[leqno]{amsmath}
\usepackage{enumerate}
\usepackage{calc}
\usepackage{amstext,amsbsy,amsopn,amsthm,amsgen,amsfonts,amscd,amsxtra,upref}
\usepackage{mathrsfs}\usepackage{euscript}\usepackage{amssymb}

\swapnumbers
\theoremstyle{plain}
\newtheorem{Thm}{Theorem}[section]
\newtheorem{Lem}[Thm]{Lemma}
\newtheorem{Cor}[Thm]{Corollary}
\newtheorem{Pro}[Thm]{Proposition}
\newtheorem{Prp}[Thm]{Properties}
\newtheorem{Sub}[Thm]{Sublemma}

\theoremstyle{definition}
\newtheorem{Def}[Thm]{Definition}
\newtheorem{Exm}[Thm]{Example}
\newtheorem{Exs}[Thm]{Examples}

\theoremstyle{remark}
\newtheorem{Rem}[Thm]{Remark}
\newtheorem{Rms}[Thm]{Remarks}
\newtheorem*{Com}{Commentary}


%% file: defin.tex
\newcommand{\myEmail}{piotr.niemiec@uj.edu.pl}
\newcommand{\myAddress}{\noindent{}Piotr Niemiec\\{}Jagiellonian University\\{}Institute of Mathematics\\{}
   ul. \L{}ojasiewicza 6\\{}30-348 Krak\'{o}w\\{}Poland}
\newcommand{\myData}{\author[P. Niemiec]{Piotr Niemiec}\address{\myAddress}\email{\myEmail}}

\newcommand{\RRR}{\mathbb{R}}

\newcommand{\ZZZ}{\mathbb{Z}}
\newcommand{\BBb}{\CMcal{B}}\newcommand{\CCc}{\CMcal{C}}
\newcommand{\EEe}{\CMcal{E}}\newcommand{\FFf}{\CMcal{F}}\newcommand{\GGg}{\CMcal{G}}\newcommand{\HHh}{\CMcal{H}}
\newcommand{\KKk}{\CMcal{K}}\newcommand{\LLl}{\CMcal{L}}
\newcommand{\NNn}{\CMcal{N}}\newcommand{\OOo}{\CMcal{O}}
\newcommand{\RRr}{\CMcal{R}}
\newcommand{\UUu}{\CMcal{U}}\newcommand{\WWw}{\CMcal{W}}\newcommand{\XXx}{\CMcal{X}}

\newcommand{\NnN}{\EuScript{N}}
\newcommand{\RrR}{\EuScript{R}}
\newcommand{\SsS}{\EuScript{S}}

\newcommand{\Bb}{\mathfrak{B}}\newcommand{\Dd}{\mathfrak{D}}

\newcommand{\bB}{\mathfrak{b}}

\newcommand{\mM}{\mathfrak{m}}

\newcommand{\SECT}[1]{\section{#1}\renewcommand{\theequation}{\thesection-\arabic{equation}}\setcounter{equation}{0}}

\newcounter{help}
\newcommand{\ITE}[3]{\ifthenelse{#1}{#2}{#3}}\newcommand{\ITEE}[3]{\ITE{\equal{#1}{#2}}{#3}{}}

\newcommand{\Ind}{\operatorname{Ind}}
\newcommand{\ind}{\operatorname{ind}}

\newcommand{\scalarr}{\langle\cdot,\mathrm{-}\rangle}
\newcommand{\scalar}[2]{\langle #1,#2\rangle}\newcommand{\leqsl}{\leqslant}\newcommand{\geqsl}{\geqslant}

\newcommand{\epsi}{\varepsilon}\newcommand{\dd}{\colon}

\newcommand{\tfcae}{the following conditions are equivalent:}

\newcommand{\THM}[1]{Theorem~\textup{\ref{thm:#1}}}
\newcommand{\COR}[1]{Corollary~\textup{\ref{cor:#1}}}
\newcommand{\LEM}[1]{Lemma~\textup{\ref{lem:#1}}}\newcommand{\PRO}[1]{Proposition~\textup{\ref{pro:#1}}}
\newcommand{\REM}[1]{Remark~\textup{\ref{rem:#1}}}

\newenvironment{thm}[1]{\begin{Thm}\label{thm:#1}}{\end{Thm}}\newenvironment{lem}[1]{\begin{Lem}\label{lem:#1}}{\end{Lem}}
\newenvironment{cor}[1]{\begin{Cor}\label{cor:#1}}{\end{Cor}}\newenvironment{pro}[1]{\begin{Pro}\label{pro:#1}}{\end{Pro}}
\newenvironment{dfn}[1]{\begin{Def}\label{def:#1}}{\end{Def}}
\newenvironment{exm}[1]{\begin{Exm}\label{exm:#1}}{\end{Exm}}
\newenvironment{rem}[1]{\begin{Rem}\label{rem:#1}}{\end{Rem}}

\input{bibl}


%% file: bibl.tex
\newcommand{\bibITEM}[2]{\ITE{\equal{#2}{}}{\bibitem{#1} }{\bibitem[#2]{#1} }}
\newcommand{\BIB}[8]{
   \bibITEM{#1}{#8} #2, \textit{#3}, #4{} \textbf{#5} (#6), #7.}
\newcommand{\myBIB}[6][P. Niemiec]{#1, \textit{#2}, #3{}\ITE{\equal{#4}{}}{}{ \textbf{#4}} (#5), #6.}
\newcommand{\BIb}[6]{
   \bibITEM{#1}{#6} #2, \textit{#3}, #4, #5.}
\newcommand{\BiB}[9]{
   \bibITEM{#1}{#9} #2, \textit{#3}, #4{} \textit{#5}, #6, #7, #8.}

\newcommand{\myBAPP}[3][P. Niemiec]{
   #1, \textit{#2}, #3}

\newcommand{\jRN}[2][]{
   \ITEE{#2}{ActaM}{\ITE{\equal{#1}{+}}
      {Acta Mathematica}{Acta Math.}}
   \ITEE{#2}{ActaMSinES}{\ITE{\equal{#1}{+}}
      {Acta Mathematica Sinica (English Series)}{Acta Math. Sin. (Engl. Ser.)}}
   \ITEE{#2}{AdvM}{\ITE{\equal{#1}{+}}
      {Advances in Mathematics}{Adv. in Math.}}
   \ITEE{#2}{ACS}{\ITE{\equal{#1}{+}}
      {Applied Categorical Structures}{Appl. Categor. Struct.}}
   \ITEE{#2}{ActaSM}{\ITE{\equal{#1}{+}}
      {Acta Scientiarum Mathematicarum}{Acta Sci. Math.}}
   \ITEE{#2}{AmJM}{\ITE{\equal{#1}{+}}
      {American Journal of Mathematics}{Amer. J. Math.}}
   \ITEE{#2}{AmMMon}{\ITE{\equal{#1}{+}}
      {American Mathematical Monthly}{Amer. Math. Mon.}}
   \ITEE{#2}{AnnSciEcNormSupT}{\ITE{\equal{#1}{+}}
      {Annales Scientifiques de l'\'{E}cole Normale Sup\'{e}rieure (3)}{Ann. Sci. \'{E}c. Norm. Sup\'{e}r. (3)}}
   \ITEE{#2}{AnnM}{\ITE{\equal{#1}{+}}
      {Annals of Mathematics}{Ann. Math.}}
   \ITEE{#2}{AnnProb}{\ITE{\equal{#1}{+}}
      {The Annals of Probability}{Ann. Probab.}}
   \ITEE{#2}{AnnPALog}{\ITE{\equal{#1}{+}}
      {Annals of Pure and Applied Logic}{Ann. Pure Appl. Logic}}
   \ITEE{#2}{APM}{\ITE{\equal{#1}{+}}
      {Annales Polonici Mathematici}{Ann. Polon. Math.}}
   \ITEE{#2}{ArchM}{\ITE{\equal{#1}{+}}
      {Archiv der Mathematik}{Arch. Math.}}
   \ITEE{#2}{AttiAccLincRendNat}{\ITE{\equal{#1}{+}}
      {Atti della Accademia Nazionale dei Lincei. Rendiconti. Classe di Scienze Fisiche, Matematiche e Naturali}
      {Atti Accad. Naz. Lincei Rend. Cl. Sci. Fis. Mat. Nat.}}
   \ITEE{#2}{BAMS}{\ITE{\equal{#1}{+}}
      {Bulletin of the American Mathematical Society}{Bull. Amer. Math. Soc.}}
   \ITEE{#2}{BAustrMS}{\ITE{\equal{#1}{+}}
      {Bulletin of the Australian Mathematical Society}{Bull. Austral. Math. Soc.}}
   \ITEE{#2}{BLondMS}{\ITE{\equal{#1}{+}}
      {Bulletin of the London Mathematical Sociecy}{Bull. Lond. Math. Soc.}}
   \ITEE{#2}{BAPolSSSM}{\ITE{\equal{#1}{+}}
      {Bulletin de l'Acad\'{e}mie Polonaise des Sciences. S\'{e}rie des Sciences Math\'{e}matiques}
      {Bull. Acad. Pol. Sci. S\'{e}r. Sci. Math.}}
   \ITEE{#2}{BullSM}{\ITE{\equal{#1}{+}}
      {Bulletin des Sciences Math\'{e}matiques}{Bull. Sci. Math.}}
   \ITEE{#2}{BullPol}{\ITE{\equal{#1}{+}}
      {Bulletin of the Polish Academy of Sciences: Mathematics}{Bull. Pol. Acad. Sci. Math.}}
   \ITEE{#2}{CanadJM}{\ITE{\equal{#1}{+}}
      {Canadian Journal Mathematics}{Canad. J. Math.}}
   \ITEE{#2}{CollectM}{\ITE{\equal{#1}{+}}
      {Collectanea Mathematica}{Collect. Math.}}
   \ITEE{#2}{CMUC}{\ITE{\equal{#1}{+}}
      {Commentationes Mathematicae Universitatis Carolinae}{Comment. Math. Univ. Carolin.}}
   \ITEE{#2}{CRParis}{\ITE{\equal{#1}{+}}
      {C. R. Paris}{C. R. Paris}}
   \ITEE{#2}{CRASParis}{\ITE{\equal{#1}{+}}
      {Comptes Rendus de l'Acad\'{e}mie des Sciences. Paris}{C. R. Acad. Sci. Paris}}
   \ITEE{#2}{CEurJM}{\ITE{\equal{#1}{+}}
      {Central European Journal of Mathematics}{Cent. Eur. J. Math.}}
   \ITEE{#2}{CMHelv}{\ITE{\equal{#1}{+}}
      {Commentarii Mathematici Helvetici}{Comment. Math. Helv.}}
   \ITEE{#2}{CollM}{\ITE{\equal{#1}{+}}
      {Colloquium Mathematicum}{Coll. Math.}}
   \ITEE{#2}{ComposM}{\ITE{\equal{#1}{+}}
      {Compositio Mathematica}{Compos. Math.}}
   \ITEE{#2}{CzMJ}{\ITE{\equal{#1}{+}}
      {Czechoslovak Mathematical Journal}{Czech. Math. J.}}
   \ITEE{#2}{DissM}{\ITE{\equal{#1}{+}}
      {Dissertationes Mathematicae}{Dissert. Math.}}
   \ITEE{#2}{DANSSSR}{\ITE{\equal{#1}{+}}
      {Doklady Akademii Nauk SSSR}{Dokl. Akad. Nauk SSSR}}
   \ITEE{#2}{DukeMJ}{\ITE{\equal{#1}{+}}
      {Duke Mathematical Journal}{Duke Math. J.}}
   \ITEE{#2}{ELA}{\ITE{\equal{#1}{+}}
      {The Electronic Journal of Linear Algebra}{Electron. J. Linear Algebra}}
   \ITEE{#2}{ExtrM}{\ITE{\equal{#1}{+}}
      {Extracta Mathematicae}{Extracta Math.}}
   \ITEE{#2}{FM}{\ITE{\equal{#1}{+}}
      {Fundamenta Mathematicae}{Fund. Math.}}
   \ITEE{#2}{FAA}{\ITE{\equal{#1}{+}}
      {Functional Analysis and its Applications}{Funct. Anal. Appl.}}
   \ITEE{#2}{FunkAnalPril}{\ITE{\equal{#1}{+}}
      {Funktsional'ny\u{\i} Analiz i Ego Prilozheniya}{Funkts. Anal. Prilozh.}}
   \ITEE{#2}{GTopA}{\ITE{\equal{#1}{+}}
      {General Topology and its Applications}{General Topol. Appl.}}
   \ITEE{#2}{HJM}{\ITE{\equal{#1}{+}}
      {Houston Journal of Mathematics}{Houston J. Math.}}
   \ITEE{#2}{IllinoisJM}{\ITE{\equal{#1}{+}}
      {Illinois Journal of Mathematics}{Illinois J. Math.}}
   \ITEE{#2}{IndagMP}{\ITE{\equal{#1}{+}}
      {Indagationes Mathematicae (Proceedings)}{Indagationes Math. Proc.}}
   \ITEE{#2}{IndianaUMJ}{\ITE{\equal{#1}{+}}
      {Indiana University Mathematical Journal}{Indiana Univ. Math. J.}}
   \ITEE{#2}{InHauEtSPM}{\ITE{\equal{#1}{+}}
      {Inst. Hautes \'{E}tudes Sci. Publ. Math.}{Inst. Hautes \'{E}tudes Sci. Publ. Math.}}
   \ITEE{#2}{IEOT}{\ITE{\equal{#1}{+}}
      {Integral Equations and Operator Theory}{Integr. Equ. Oper. Theory}}
   \ITEE{#2}{IsraelJM}{\ITE{\equal{#1}{+}}
      {Israel Journal of Mathematics}{Israel J. Math.}}
   \ITEE{#2}{JAusMSA}{\ITE{\equal{#1}{+}}
      {Journal of the Australian Mathematical Society. Series A}{J. Aust. Math. Soc. Ser. A}}
   \ITEE{#2}{JCA}{\ITE{\equal{#1}{+}}
      {Journal of Convex Analysis}{J. Convex Anal.}}
   \ITEE{#2}{JChinUST}{\ITE{\equal{#1}{+}}
      {J. China Univ. Sci. Tech.}{J. China Univ. Sci. Tech.}}
   \ITEE{#2}{JFA}{\ITE{\equal{#1}{+}}
      {Journal of Functional Analysis}{J. Funct. Anal.}}
   \ITEE{#2}{JKoreanMS}{\ITE{\equal{#1}{+}}
      {Journal of the Korean Mathematical Society}{J. Korean Math. Soc.}}
   \ITEE{#2}{JMAnApp}{\ITE{\equal{#1}{+}}
      {J. Math. Anal. Appl.}{J. Math. Anal. Appl.}}
   \ITEE{#2}{JOT}{\ITE{\equal{#1}{+}}
      {Journal of Operator Theory}{J. Operator Theory}}
   \ITEE{#2}{KodaiMSemRep}{\ITE{\equal{#1}{+}}
      {Kodai Math. Sem. Rep.}{Kodai Math. Sem. Rep.}}
   \ITEE{#2}{LAA}{\ITE{\equal{#1}{+}}
      {Linear Algebra and its Applications}{Linear Algebra Appl.}}
   \ITEE{#2}{LMLA}{\ITE{\equal{#1}{+}}
      {Linear and Multilinear Algebra}{Linear Multilinear Algebra}}
   \ITEE{#2}{LNM}{\ITE{\equal{#1}{+}}
      {Lecture Notes in Mathematics}{Lecture Notes Math.}}
   \ITEE{#2}{MathJap}{\ITE{\equal{#1}{+}}
      {Math. Japon.}{Math. Japon.}}
   \ITEE{#2}{MLQ}{\ITE{\equal{#1}{+}}
      {Mathematical Logic Quarterly}{Math. Log. Q.}}
   \ITEE{#2}{MProcCambPhS}{\ITE{\equal{#1}{+}}
      {Mathematical Proceedings of the Cambridge Philosophical Society}{Math. Proc. Cambridge Phil. Soc.}}
   \ITEE{#2}{MMag}{\ITE{\equal{#1}{+}}
      {Mathematics Magazine}{Math. Mag.}}
   \ITEE{#2}{MSb}{\ITE{\equal{#1}{+}}
      {Matematicheski\u{\i} Sbornik}{Mat. Sb.}}
   \ITEE{#2}{MStud}{\ITE{\equal{#1}{+}}
      {Matematychni Studi\"{\i}}{Mat. Stud.}}
   \ITEE{#2}{MScand}{\ITE{\equal{#1}{+}}
      {Mathematica Scandinavica}{Math. Scand.}}
   \ITEE{#2}{MAnn}{\ITE{\equal{#1}{+}}
      {Mathematische Annalen}{Math. Ann.}}
   \ITEE{#2}{MAMS}{\ITE{\equal{#1}{+}}
      {Memoirs of the American Mathematical Society}{Mem. Amer. Math. Soc.}}
   \ITEE{#2}{MichMJ}{\ITE{\equal{#1}{+}}
      {Michigan Mathematical Journal}{Mich. Math. J.}}
   \ITEE{#2}{MonatM}{\ITE{\equal{#1}{+}}
      {Monatshefte f\"{u}r Mathematik}{Mh. Math.}}
   \ITEE{#2}{NonlinA}{\ITE{\equal{#1}{+}}
      {Nonlinear Analysis: Theory, Methods \& Applications}{Nonlinear Anal.}}
   \ITEE{#2}{NAMS}{\ITE{\equal{#1}{+}}
      {Notices of the American Mathematical Society}{Notices Amer. Math. Soc.}}
   \ITEE{#2}{OpusM}{\ITE{\equal{#1}{+}}
      {Opuscula Mathematica}{Opuscula Math.}}
   \ITEE{#2}{PacJM}{\ITE{\equal{#1}{+}}
      {Pacific Journal of Mathematics}{Pacific J. Math.}}
   \ITEE{#2}{PeriodMHung}{\ITE{\equal{#1}{+}}
      {Periodica Mathematica Hungarica}{Period. Math. Hungarica}}
   \ITEE{#2}{PAMS}{\ITE{\equal{#1}{+}}
      {Proceedings of the American Mathematical Society}{Proc. Amer. Math. Soc.}}
   \ITEE{#2}{ProcCambPhS}{\ITE{\equal{#1}{+}}
      {Proceedings of the Cambridge Philosophical Society}{Proc. Cambridge Phil. Soc.}}
   \ITEE{#2}{ProcImpAcadTokyo}{\ITE{\equal{#1}{+}}
      {Proc. Imp. Acad. Tokyo}{Proc. Imp. Acad. Tokyo}}
   \ITEE{#2}{ProcKonink}{\ITE{\equal{#1}{+}}
      {Proceedings of the Koninklijke Nederlandse Akademie van Wetenschappen}{Nederl. Akad. Wetensch. Proc. Ser. A}}
   \ITEE{#2}{PLondMS}{\ITE{\equal{#1}{+}}
      {Proceedings of the London Mathematical Society}{Proc. London Math. Soc.}}
   \ITEE{#2}{PNatlUSA}{\ITE{\equal{#1}{+}}
      {Proceedings of the National Academy of Sciences of the United States of America}{Proc. Natl. Acad. Sci. USA}}
   \ITEE{#2}{PublRIMSKyoto}{\ITE{\equal{#1}{+}}
      {Publ. Res. Inst. Math. Sci. Kyoto Univ.}{Publ. Res. Inst. Math. Sci.}}
   \ITEE{#2}{PWN}{\ITE{\equal{#1}{+}}
      {PWN -- Polish Scientific Publishers, Warszawa}{PWN -- Polish Scientific Publishers, Warszawa}}
   \ITEE{#2}{RCMP}{\ITE{\equal{#1}{+}}
      {Rendiconti del Circolo Matematico di Palermo}{Rend. Circ. Mat. Palermo}}
   \ITEE{#2}{RussMS}{\ITE{\equal{#1}{+}}
      {Russian Mathematical Surveys}{Russian Math. Surveys}}
   \ITEE{#2}{SbM}{\ITE{\equal{#1}{+}}
      {Sbornik: Mathematics}{Sb. Math.}}
   \ITEE{#2}{SciRepTokyoA}{\ITE{\equal{#1}{+}}
      {Science Reports of Tokyo Kyoiku Daigaku, Section A}{Sci. Rep. Tokyo Kyoiku Daigaku Sect. A}}
   \ITEE{#2}{SeminProbStras}{\ITE{\equal{#1}{+}}
      {S\'{e}minaire de probabilit\'{e}s de Strasbourg}{S\'{e}min. Prob. Strasbourg}}
   \ITEE{#2}{SIAMJMAA}{\ITE{\equal{#1}{+}}
      {SIAM Journal on Matrix Analysis and Applications}{SIAM J. Matrix Anal. Appl.}}
   \ITEE{#2}{SibirMZ}{\ITE{\equal{#1}{+}}
      {Sibirski\v{\i} Mat. \v{Z}hurnal}{Sibirsk. Mat. \v{Z}.}}
   \ITEE{#2}{SM}{\ITE{\equal{#1}{+}}
      {Studia Mathematica}{Studia Math.}}
   \ITEE{#2}{TAMS}{\ITE{\equal{#1}{+}}
      {Transactions of the American Mathematical Society}{Trans. Amer. Math. Soc.}}
   \ITEE{#2}{TohokuMJ}{\ITE{\equal{#1}{+}}
      {T\^{o}hoku Mathematical Journal}{T\^{o}hoku Math. J.}}
   \ITEE{#2}{TomskUnivRev}{\ITE{\equal{#1}{+}}
      {Tomsk Universitet Review}{Tomsk. Univ. Rev.}}
   \ITEE{#2}{TopA}{\ITE{\equal{#1}{+}}
      {Topology and its Applications}{Topology Appl.}}
   \ITEE{#2}{TopMethNA}{\ITE{\equal{#1}{+}}
      {Topological Methods in Nonlinear Analysis}{Topol. Methods Nonlinear Anal.}}
   \ITEE{#2}{TsukubaJM}{\ITE{\equal{#1}{+}}
      {Tsukuba Journal of Mathematics}{Tsukuba J. Math.}}
   \ITEE{#2}{UspekhiMN}{\ITE{\equal{#1}{+}}
      {Uspekhi Matem. Nauk}{Uspekhi Mat. Nauk}}
   }

\newcommand{\paplist}[3][]{
   \ITEE{#3}{NIAkhiezer,IMGlazman1993}{
      \BIb{#2}{N.I. Akhiezer and I.M. Glazman}
         {Theory of Linear Operators in Hilbert Space}
         {Dover Publications, Inc., New York}{1993}{#1}}
   \ITEE{#3}{RDAnderson1967}{
      \BIB{#2}{R.D. Anderson}
         {On topological infinite deficiency}
         {\jRN{MichMJ}}{14}{1967}{365--383}{#1}}
   \ITEE{#3}{RDAnderson,JMcCharen1970}{
      \BIB{#2}{R.D. Anderson and J. McCharen}
         {On extending homeomorphisms to Fr\'{e}chet manifolds}
         {\jRN{PAMS}}{25}{1970}{283--289}{#1}}
   \ITEE{#3}{RDAnderson,DWCurtis,JVanMill1982}{
      \BIB{#2}{R.D. Anderson, D.W. Curtis, J. van Mill}
         {A fake topological Hilbert space}
         {\jRN{TAMS}}{272}{1982}{311--321}{#1}}
   \ITEE{#3}{RArens,JEells1956}{
      \BIB{#2}{R. Arens and J. Eells}
         {On embedding uniform and topological spaces}
         {\jRN{PacJM}}{6}{1956}{397--403}{#1}}
   \ITEE{#3}{NAronszajn,PPanitchpakdi1956}{
      \BIB{#2}{N. Aronszajn and P. Panitchpakdi}
         {Extension of uniformly continuous transformations and hyperconvex metric spaces}
         {\jRN{PacJM}}{6}{1956}{405--439}{#1}}
   \ITEE{#3}{KJBabenko1948}{
      \BIB{#2}{K.J. Babenko}
         {On conjugate functions}
         {\jRN{DANSSSR}}{62}{1948}{157--160}{#1}}
   \ITEE{#3}{TBanakh1995}{
      \BIB{#2}{T.O. Banakh}
         {Topology of spaces of probability measures, I}
         {\jRN{MStud}}{5}{1995}{65--87 (Russian)}{#1}}
   \ITEE{#3}{TBanakh1995a}{
      \BIB{#2}{T.O. Banakh}
         {Topology of spaces of probability measures, II}
         {\jRN{MStud}}{5}{1995}{88--106 (Russian)}{#1}}
   \ITEE{#3}{TBanakh1998}{
      \BIB{#2}{T. Banakh}
         {Characterization of spaces admitting a homotopy dense embedding into a Hilbert manifold}
         {\jRN{TopA}}{86}{1998}{123--131}{#1}}
   \ITEE{#3}{TBanakh,TNRadul1997}{
      \BIB{#2}{T.O. Banakh and T.N. Radul}
         {Topology of spaces of probability measures}
         {\jRN{SbM}}{188}{1997}{973--995}{#1}}
   \ITEE{#3}{TBanakh,TRadul,MZarichnyi1996}{
      \BIb{#2}{T. Banakh, T. Radul, M. Zarichnyi}
         {Absorbing sets in infinite-dimensional manifolds}
         {VNTL Publishers, Lviv}{1996}{#1}}
   \ITEE{#3}{TBanakh,IZarichnyy2008}{
      \BIB{#2}{T. Banakh and I. Zarichnyy}
         {Topological groups and convex sets homeomorphic to non-separable Hilbert spaces}
         {\jRN{CEurJM}}{6}{2008}{77--86}{#1}}
   \ITEE{#3}{HBecker,ASKechris1996}{
      \BIb{#2}{H. Becker and A.S. Kechris}
         {The Descriptive Set Theory of Polish Group Actions \textup{(London Math. Soc. Lecture Note Series, vol. 232)}}
         {University Press, Cambridge}{1996}{#1}}
   \ITEE{#3}{GBeer1993}{
      \BIb{#2}{G. Beer}
         {Topologies on Closed and Closed Convex Sets \textup{(Mathematics and Its Applications)}}
         {Kluwer Academic Publishers, Dordrecht}{1993}{#1}}
   \ITEE{#3}{NEBenamara,NNikolski1999}{
      \BIB{#2}{N.E. Benamara and N. Nikolski}
         {Resolvent tests for similarity to a normal operator}
         {\jRN{PLondMS}}{78}{1999}{585--626}{#1}}
   \ITEE{#3}{YBenyamini,JLindenstrauss2000}{
      \BIb{#2}{Y. Benyamini and J. Lindenstrauss}
         {Geometric nonlinear functional analysis I}
         {AMS Colloquium Publications 48}{2000}{#1}}
   \ITEE{#3}{SKBerberian1974}{
      \BIb{#2}{S.K. Berberian}
         {Lectures in Functional Analysis and Operator Theory}
         {Graduate Texts in Mathematics 15, Springer-Verlag, New York}{1974}{#1}}
   \ITEE{#3}{SNBernstein1954}{
      \BIb{#2}{S.N. Bernstein}
         {Collected Works II}
         {Akad. Nauk SSSR, Moscow}{1954 (Russian)}{#1}}
   \ITEE{#3}{CzBessaga,APelczynski1972}{
      \BIB{#2}{Cz. Bessaga and A. Pe\l{}czy\'{n}ski}
         {On spaces of measurable functions}
         {\jRN{SM}}{44}{1972}{597--615}{#1}}
   \ITEE{#3}{CzBessaga,APelczynski1975}{
      \BIb{#2}{Cz. Bessaga and A. Pe\l{}czy\'{n}ski}
         {Selected topics in infinite-dimensional topology}
         {\jRN{PWN}}{1975}{#1}}
   \ITEE{#3}{MBestvina,JMogilski1986}{
      \BIB{#2}{M. Bestvina and J. Mogilski}
         {Characterizing certain incomplete infinite-dimensional absolute retracts}
         {\jRN{MichMJ}}{33}{1986}{291--313}{#1}}
   \ITEE{#3}{MBestvina,PBowers,JMogilsky,JWalsh1986}{
      \BIB{#2}{M. Bestvina, P. Bowers, J. Mogilsky, J. Walsh}
         {Characterization of Hilbert space manifolds revisited}
         {\jRN{TopA}}{24}{1986}{53--69}{#1}}
   \ITEE{#3}{RBhatia1997}{
      \BIb{#2}{R. Bhatia}
         {Matrix Analysis}
         {Springer, New York}{1997}{#1}}
   \ITEE{#3}{GBirkhoff1936}{
      \BIB{#2}{G. Birkhoff}
         {A note on topological groups}
         {\jRN{ComposM}}{3}{1936}{427--430}{#1}}
   \ITEE{#3}{MSBirman,MZSolomjak1987}{
      \BIb{#2}{M.S. Birman and M.Z. Solomjak}
         {Spectral Theory of Self-Adjoint Operators in Hilbert Space}
         {D. Reidel Publishing Co., Dordrecht}{1987}{#1}}
   \ITEE{#3}{EBishop1961}{
      \BIB{#2}{E. Bishop}
         {A generalization of the Stone-Weierstrass theorem}
         {\jRN{PacJM}}{11}{1961}{777--783}{#1}}
   \ITEE{#3}{BBlackadar2006}{
      \BIb{#2}{B. Blackadar}{Operator Algebras. Theory of $\CCc^*$-algebras and von Neumann algebras 
         \textup{(Encyclopaedia of Mathematical Sciences, vol. 122: Operator Algebras and Non-Commutative Geometry III)}}
         {Springer-Verlag, Berlin-Heidelberg}{2006}{#1}}
   \ITEE{#3}{JBlass,WHolsztynski1972}{
      \BIB{#2}{J. Blass and W. Holszty\'{n}ski}
         {Cubical polyhedra and homotopy III}
         {\jRN{AttiAccLincRendNat}}{53}{1972}{275--279}{#1}}
   \ITEE{#3}{FFBonsall,NJDuncan1973}{
      \BIb{#2}{F.F. Bonsall and N.J. Duncan}
         {Complete Normed Algebras}
         {Springer Verlag, Berlin}{1973}{#1}}
   \ITEE{#3}{NBourbaki2002}{
      \BIb{#2}{N. Bourbaki}
         {Lie Groups and Lie Algebras, Chapters 4--6}
         {Springer, New York}{2002}{#1}}
   \ITEE{#3}{PLBowers1989}{
      \BIB{#2}{P.L. Bowers}
         {Limitation topologies on function spaces}
         {\jRN{TAMS}}{314}{1989}{421--431}{#1}}
   \ITEE{#3}{ABrown1953}{
      \BIB{#2}{A. Brown}
         {On a class of operators}
         {\jRN{PAMS}}{4}{1953}{723--728}{#1}}
   \ITEE{#3}{ABrown,CKFong,DWHadwin1978}{
      \BIB{#2}{A. Brown, C.-K. Fong, D.W. Hadwin}
         {Parts of operators on Hilbert space}
         {\jRN{IllinoisJM}}{22}{1978}{306--314}{#1}}
   \ITEE{#3}{AMBruckner,JBBruckner,BSThomson1997}{
      \BIb{#2}{A.M. Bruckner, J.B. Bruckner, B.S. Thomson}
         {Real Analysis}
         {Prentice-Hall, New Jersey}{1997}{#1}}
   \ITEE{#3}{PJCameron,AMVershik2006}{
      \BIB{#2}{P.J. Cameron and A.M. Vershik}
         {Some isometry groups of Urysohn space}
         {\jRN{AnnPALog}}{143}{2006}{70--78}{#1}}
   \ITEE{#3}{CCastaing1966}{
      \BIB{#2}{C. Castaing}
         {Quelques probl\`{e}mes de mesurabilit\'{e} li\'{e}es \`{a} la th\'{e}orie de la commande}
         {\jRN{CRParis}}{262}{1966}{409--411}{#1}}
   \ITEE{#3}{JAVanCasteren1980}{
      \BIB{#2}{J.A. van Casteren}
         {A problem of Sz.-Nagy}
         {\jRN{ActaSM}}{42}{1980}{189--194}{#1}}
   \ITEE{#3}{JAVanCasteren1983}{
      \BIB{#2}{J.A. van Casteren}
         {Operators similar to unitary or selfadjoint ones}
         {\jRN{PacJM}}{104}{1983}{241--255}{#1}}
   \ITEE{#3}{XCatepillan,MPtak,WSzymanski1994}{
      \BIB{#2}{X. Catepill\'{a}n, M. Ptak, W. Szyma\'{n}ski}
         {Multiple canonical decompositions of families of operators and a model of quasinormal families}
         {\jRN{PAMS}}{121}{1994}{1165--1172}{#1}}
   \ITEE{#3}{RCauty1994}{
      \BIB{#2}{R. Cauty}
         {Un espace m\'{e}trique lin\'{e}aire qui n'est pas un r\'{e}tracte absolu}
         {\jRN{FM}}{146}{1994}{85--99, (French)}{#1}}
   \ITEE{#3}{TAChapman1971}{
      \BIB{#2}{T.A. Chapman}
         {Deficiency in infinite-dimensional manifolds}
         {\jRN{GTopA}}{1}{1971}{263--272}{#1}}
   \ITEE{#3}{TAChapman1976}{
      \BIb{#2}{T.A. Chapman}
         {Lectures on Hilbert cube manifolds}
         {C.B.M.S. Regional Conference Series in Math. No 28, Amer. Math. Soc.}{1976}{#1}}
   \ITEE{#3}{RBChuaqui1977}{
      \BIB{#2}{R.B. Chuaqui}
         {Measures invariant under a group of transformations}
         {\jRN{PacJM}}{68}{1977}{313--329}{#1}}
   \ITEE{#3}{JBConway1985}{
      \BIb{#2}{J.B. Conway}
         {A Course in Functional Analysis}
         {Springer-Verlag, New York}{1985}{#1}}
   \ITEE{#3}{JBConway2000}{
      \BIb{#2}{J.B. Conway}
         {A Course in Operator Theory}
         {(Graduate Studies in Mathematics, vol. 21) Amer. Math. Soc., Providence}{2000}{#1}}
   \ITEE{#3}{GCorach,AMaestripieri,MMbekhta2009}{
      \BIB{#2}{G. Corach, A. Maestripieri, M. Mbekhta}
         {Metric and homogeneous structure of closed range operators}
         {\jRN{JOT}}{61}{2009}{171--190}{#1}}
   \ITEE{#3}{MJCowen,RGDouglas1978}{
      \BIB{#2}{M.J. Cowen and R.G. Douglas}
         {Complex geometry and operator theory}
         {\jRN{ActaM}}{141}{1978}{187--261}{#1}}
   \ITEE{#3}{DWCurtis1985}{
      \BIB{#2}{D.W. Curtis}
         {Boundary sets in the Hilbert cube}
         {\jRN{TopA}}{20}{1985}{201--221}{#1}}
   \ITEE{#3}{MMDay1958}{
      \BIb{#2}{M.M. Day}
         {Normed Linear Spaces}
         {Springer Verlag, Berlin}{1958}{#1}}
   \ITEE{#3}{CDellacherie1967}{
      \BIB{#2}{C. Dellacherie}
         {Un compl\'{e}ment au th\'{e}or\`{e}me de Weierstrass-Stone}
         {\jRN{SeminProbStras}}{1}{1967}{52--53}{#1}}
   \ITEE{#3}{JJDijkstra1987}{
      \BIB{#2}{J.J. Dijkstra}
         {Strong negligibility of $\sigma$-compacta does not characterize Hilbert space}
         {\jRN{PacJM}}{127}{1987}{19--30}{#1}}
   \ITEE{#3}{JJDijkstra1990}{
      \BIB{#2}{J.J. Dijkstra}
         {Characterizing Hilbert space topology in terms of strong negligibility}
         {\jRN{ComposM}}{75}{1990}{299--306}{#1}}
   \ITEE{#3}{TDobrowolski,WMarciszewski2002}{
      \BIB{#2}{T. Dobrowolski and W. Marciszewski}
         {Failure of the Factor Theorem for Borel pre-Hilbert spaces}
         {\jRN{FM}}{175}{2002}{53--68}{#1}}
   \ITEE{#3}{TDobrowolski,JMogilski1990}{
      \BiB{#2}{T. Dobrowolski and J. Mogilski}
         {Problems on Topological Classification of Incomplete Metric Spaces}{Chapter 25 in:}
         {Open Problems in Topology}{J. van Mill and G.M. Reed (eds.), North-Holland Amsterdam}{1990}{411--429}{#1}}
   \ITEE{#3}{TDobrowolski,HTorunczyk1981}{
      \BIB{#2}{T. Dobrowolski and H. Toru\'{n}czyk}
         {Separable complete ANR's admitting a group structure are Hilbert manifolds}
         {\jRN{TopA}}{12}{1981}{229--235}{#1}}
   \ITEE{#3}{RGDouglas1966}{
      \BIB{#2}{R.G. Douglas}
         {On majorization, factorization and range inclusion of operators in Hilbert space}
         {\jRN{PAMS}}{17}{1966}{413--416}{#1}}
   \ITEE{#3}{CHDowker1947}{
      \BIB{#2}{C.H. Dowker}
         {Mapping theorems for non-compact spaces}
         {\jRN{AmJM}}{69}{1947}{200--242}{#1}}
   \ITEE{#3}{CHDowker1952}{
      \BIB{#2}{C.H. Dowker}
         {Topology of metric complexes}
         {\jRN{AmJM}}{74}{1952}{555--577}{#1}}
   \ITEE{#3}{JDugundji1951}{
      \BIB{#2}{J. Dugundji}
         {An extension of Tietze's theorem}
         {\jRN{PacJM}}{1}{1951}{353--367}{#1}}
   \ITEE{#3}{JDugundji1958}{
      \BIB{#2}{J. Dugundji}
         {Absolute neighborhood retracts and local connectedness for arbitrary metric spaces}
         {\jRN{ComposM}}{13}{1958}{229--246}{#1}}
   \ITEE{#3}{JDugundji1965}{
      \BIB{#2}{J. Dugundji}
         {Locally equiconnected spaces and absolute neighborhood retracts}
         {\jRN{FM}}{57}{1965}{187--193}{#1}}
   \ITEE{#3}{NDunford,JTSchwartz1958}{
      \BIb{#2}{N. Dunford and J.T. Schwartz}
         {Linear Operators, part I}
         {Interscience Publishers, New York}{1958}{#1}}
   \ITEE{#3}{NDunford,JTSchwartz1963}{
      \BIb{#2}{N. Dunford and J.T. Schwartz}
         {Linear Operators, part II}
         {Interscience Publishers, New York}{1963}{#1}}
   \ITEE{#3}{NDunford,JTSchwartz1971}{
      \BIb{#2}{N. Dunford and J.T. Schwartz}
         {Linear Operators, part III}
         {Wiley-Interscience, New York}{1971}{#1}}
   \ITEE{#3}{MLEaton,MDPerlman1977}{
      \BIB{#2}{M.L. Eaton and M.D. Perlman}
         {Reflection groups, generalized Schur functions and the geometry of majorization}
         {\jRN{AnnProb}}{5}{1977}{829--860}{#1}}
   \ITEE{#3}{EGEffros1965}{
      \BIB{#2}{E.G. Effros}
         {The Borel space of von Neumann algebras on a separable Hilbert space}
         {\jRN{PacJM}}{15}{1965}{1153--1164}{#1}}
   \ITEE{#3}{EGEffros1966}{
      \BIB{#2}{E.G. Effros}
         {Global structure in von Neumann algebras}
         {\jRN{TAMS}}{121}{1966}{434--454}{#1}}
   \ITEE{#3}{REspinola,MAKhamsi2001}{
      \BiB{#2}{R. Espinola and M.A. Khamsi}
         {Introduction to hyperconvex spaces}{Chapter XIII in:}{Handbook of Metric Fixed Point Theory}
         {W.A. Kirk and B. Sims (editors), Kluwer Academic Publishers}{2001}{391--435}{#1}}
   \ITEE{#3}{PAFillmore,JPWilliams1971}{
      \BIB{#2}{P.A. Fillmore and J.P. Williams}
         {On operator ranges}
         {\jRN{AdvM}}{7}{1971}{254--281}{#1}}
   \ITEE{#3}{JEells,NHKuiper1969}{
      \BIB{#2}{J. Eells and N.H. Kuiper}
         {Homotopy negligible subsets in infinite-dimensional manifolds}
         {\jRN{ComposM}}{21}{1969}{151--161}{#1}}
   \ITEE{#3}{REngelking1977}{
      \BIb{#2}{R. Engelking}
         {General Topology}
         {\jRN{PWN}}{1977}{#1}}
   \ITEE{#3}{REngelking1978}{
      \BIb{#2}{R. Engelking}
         {Dimension Theory}
         {\jRN{PWN}}{1978}{#1}}
   \ITEE{#3}{REngelking1989}{
      \BIb{#2}{R. Engelking}
         {General Topology. Revised and completed edition \textup{(Sigma series in pure mathematics, vol. 6)}}
         {Heldermann Verlag, Berlin}{1989}{#1}}
   \ITEE{#3}{PErdos,RDMauldin1976}{
      \BIB{#2}{P. Erd\"{o}s and R.D. Mauldin}
         {The nonexistence of certain invariant measures}
         {\jRN{PAMS}}{59}{1976}{321--322}{#1}}
   \ITEE{#3}{JErnest1976}{
      \BIB{#2}{J. Ernest}
         {Charting the operator terrain}
         {\jRN{MAMS}}{171}{1976}{207 pp}{#1}}
   \ITEE{#3}{RHFox1943}{
      \BIB{#2}{R.H. Fox}
         {On fiber spaces, II}
         {\jRN{BAMS}}{49}{1943}{733--735}{#1}}
   \ITEE{#3}{NAFriedman1970}{
      \BIb{#2}{N.A. Friedman}
         {Introduction to ergodic theory}
         {Van Nostrand Reinhold Company}{1970}{#1}}
   \ITEE{#3}{MFujii,MKajiwara,YKato,FKubo1976}{
      \BIB{#2}{M. Fujii, M. Kajiwara, Y. Kato, F. Kubo}
         {Decompositions of operators in Hilbert spaces}
         {\jRN{MathJap}}{21}{1976}{117--120}{#1}}
   \ITEE{#3}{SGao,ASKechris2003}{
      \BIB{#2}{S. Gao and A.S. Kechris}
         {On the classification of Polish metric spaces up to isometry}
         {\jRN{MAMS}}{161}{2003}{viii+78}{#1}}
   \ITEE{#3}{MIGarrido,FMontalvo1991}{
      \BIB{#2}{M.I. Garrido and F. Montalvo}
         {On some generalizations of the Kakutani-Stone and Stone-Weierstrass theorems}
         {\jRN{ExtrM}}{6}{1991}{156--159}{#1}}
   \ITEE{#3}{LGe,JShen2002}{
      \BIB{#2}{L. Ge and J. Shen}
         {Generator problem for certain property T factors}
         {\jRN{PNAS}}{99}{2002}{565--567}{#1}}
   \ITEE{#3}{IMGelfand,MANaimark1943}{
      \BIB{#2}{I.M. Gelfand and M.A. Naimark}
         {On the embedding of normed rings into the ring of operators in Hilbert space}
         {\jRN{MSb}}{12}{1943}{197--213}{#1}}
   \ITEE{#3}{FGesztesy,MMalamud,MMitrea,SNaboko2009}{
      \BIB{#2}{F. Gesztesy, M. Malamud, M. Mitrea, S. Naboko}
         {Generalized polar decompositions for closed operators in Hilbert spaces and some applications}
         {\jRN{IEOT}}{64}{2009}{83--113}{#1}}
   \ITEE{#3}{LGillman,MJerison1960}{
      \BIb{#2}{L. Gillman and M. Jerison}
         {Rings of continuous functions}
         {New York}{1960}{#1}}
   \ITEE{#3}{JGlimm1960}{
      \BIB{#2}{J. Glimm}
         {A Stone-Weierstrass theorem for $\CCc^*$-algebras}
         {\jRN{AnnM}}{72}{1960}{216--244}{#1}}
   \ITEE{#3}{GGodefroy,NJKalton2003}{
      \BIB{#2}{G. Godefroy and N.J. Kalton}
         {Lipschitz-free Banach spaces}
         {\jRN{SM}}{159}{2003}{121--141}{#1}}
   \ITEE{#3}{ICGohberg,MGKrein1967}{
      \BIB{#2}{I.C. Gohberg and M.G. Krein}
         {On a description of contraction operators similar to unitary ones}
         {\jRN{FunkAnalPril}}{1}{1967}{38--60}{#1}}
   \ITEE{#3}{ELGriffinJr1953}{
      \BIB{#2}{E.L. Griffin Jr.}
         {Some contributions to the theory of rings of operators}
         {\jRN{TAMS}}{75}{1953}{471--504}{#1}}
   \ITEE{#3}{ELGriffinJr1955}{
      \BIB{#2}{E.L. Griffin Jr.}
         {Some contributions to the theory of rings of operators II}
         {\jRN{TAMS}}{79}{1955}{389--400}{#1}}
   \ITEE{#3}{MGromov1981}{
      \BIB{#2}{M. Gromov}
         {Groups of polynomial growth and expanding maps}
         {\jRN{InHauEtSPM}}{53}{1981}{53--73}{#1}}
   \ITEE{#3}{MGromov1999}{
      \BIb{#2}{M. Gromov}
         {Metric Structures for Riemannian and Non-Riemannian Spaces}
         {Progress in Math. \textbf{152}, Birkh\"{a}user}{1999}{#1}}
   \ITEE{#3}{JDeGroot1956}{
      \BIB{#2}{J. de Groot}
         {Non-archimedean metrics in topology}
         {\jRN{PAMS}}{7}{1956}{948--953}{#1}}
   \ITEE{#3}{LCGrove,CTBenson1985}{
      \BIb{#2}{L.C. Grove and C.T. Benson}
         {Finite Reflection Group}
         {2nd ed., Springer-Verlag}{1985}{#1}}
   \ITEE{#3}{VIGurarii1966}{
      \BIB{#2}{V.I. Gurari\v{\i}}{Spaces of universal placement, isotropic spaces and a problem of Mazur 
         on rotations of Banach spaces \textup{(Russian)}}
         {\jRN{SibirMZ}}{7}{1966}{1002--1013}{#1}}
   \ITEE{#3}{DWHadwin1976}{
      \BIB{#2}{D.W. Hadwin}
         {An operator-valued spectrum}
         {\jRN{NAMS}}{23}{1976}{A-163}{#1}}
   \ITEE{#3}{DWHadwin1977}{
      \BIB{#2}{D.W. Hadwin}
         {An operator-valued spectrum}
         {\jRN{IndianaUMJ}}{26}{1977}{329--340}{#1}}
   \ITEE{#3}{HHahn1932}{
      \BIb{#2}{H. Hahn}
         {Reelle Funktionen I}
         {Leipzig}{1932}{#1}}
   \ITEE{#3}{PRHalmos1950}{
      \BIb{#2}{P.R. Halmos}
         {Measure theory}
         {Van Nostrand, New York}{1950}{#1}}
   \ITEE{#3}{PRHalmos1951}{
      \BIb{#2}{P.R. Halmos}
         {Introduction to Hilbert Space and the Theory of Spectral Multiplicity}
         {Chelsea Publishing Company, New York}{1951}{#1}}
   \ITEE{#3}{PRHalmos1956}{
      \BIb{#2}{P.R. Halmos}
         {Lectures on Ergodic Theory}
         {Publ. Math. Soc. Japan, Tokyo}{1956}{#1}}
   \ITEE{#3}{PRHalmos1982}{
      \BIb{#2}{P.R. Halmos}
         {A Hilbert Space Problem Book}
         {Springer-Verlag New York Inc.}{1982}{#1}}
  \ITEE{#3}{PRHalmos,JEMcLaughlin1963}{
      \BIB{#2}{P.R. Halmos and J.E. McLaughlin}
         {Partial isometries}
         {\jRN{PacJM}}{13}{1963}{585--596}{#1}}
   \ITEE{#3}{RWHansell1972}{
      \BIB{#2}{R.W. Hansell}
         {On the nonseparable theory of Borel and Souslin sets}
         {\jRN{BAMS}}{78}{1972}{236--241}{#1}}
   \ITEE{#3}{FHausdorff1930}{
      \BIB{#2}{F. Hausdorff}
         {Erweiterung einer Hom\"{o}omorphie}
         {\jRN{FM}}{16}{1930}{353--360}{#1}}
   \ITEE{#3}{FHausdorff1934}{
      \BIB{#2}{F. Hausdorff}
         {\"{U}ber innere Abbildungen}
         {\jRN{FM}}{23}{1934}{279--291}{#1}}
   \ITEE{#3}{FHausdorff1938}{
      \BIB{#2}{F. Hausdorff}
         {Erweiterung einer stetigen Abbildung}
         {\jRN{FM}}{30}{1938}{40--47}{#1}}
   \ITEE{#3}{DWHenderson1971}{
      \BIB{#2}{D.W. Henderson}
         {Corrections and extensions of two papers about infinite-dimensional manifolds}
         {\jRN{GTopA}}{1}{1971}{321--327}{#1}}
   \ITEE{#3}{DWHenderson1975}{
      \BIB{#2}{D.W. Henderson}
         {$Z$-sets in ANR's}
         {\jRN{TAMS}}{213}{1975}{205--216}{#1}}
   \ITEE{#3}{DWHenderson,RMSchori1970}{
      \BIB{#2}{D.W. Henderson and R.M. Schori}
         {Topological classification of infinite-dimensional manifolds by homotopy type}
         {\jRN{BAMS}}{76}{1970}{121--124}{#1}}
   \ITEE{#3}{DWHenderson,JEWest1970}{
      \BIB{#2}{D.W. Henderson and J.E. West}
         {Triangulated infinite-dimensional manifolds}
         {\jRN{BAMS}}{76}{1970}{655--660}{#1}}
   \ITEE{#3}{BHoffmann1979}{
      \BIB{#2}{B. Hoffmann}
         {A compact contractible topological group is trivial}
         {\jRN{ArchM}}{32}{1979}{585--587}{#1}}
   \ITEE{#3}{DHofmann2002}{
      \BIB{#2}{D. Hofmann}
         {On a generalization of the Stone-Weierstrass theorem}
         {\jRN{ACS}}{10}{2002}{569--592}{#1}}
   \ITEE{#3}{GHognas,AMukherjea1995}{
      \BIb{#2}{G. H\"ogn\"as and A. Mukherjea}
         {Probability Measures on Semigroups. Convolution Products, Random Walks, and Random Matrices}
         {Plenum Press, New York}{1995}{#1}}
   \ITEE{#3}{MRHolmes1992}{
      \BIB{#2}{M.R. Holmes}
         {The universal separable metric space of Urysohn and isometric embeddings thereof in Banach spaces}
         {\jRN{FM}}{140}{1992}{199--223}{#1}}
   \ITEE{#3}{MRHolmes2008}{
      \BIB{#2}{M.R. Holmes}
         {The Urysohn space embeds in Banach spaces in just one way}
         {\jRN{TopA}}{155}{2008}{1479--1482}{#1}}
   \ITEE{#3}{RRHolmes,TYTam1999}{
      \BIB{#2}{R.R. Holmes and T.Y. Tam}
         {Distance to the convex hull of an orbit under the action of a compact group}
         {\jRN{JAusMSA}}{66}{1999}{331--357}{#1}}
   \ITEE{#3}{RHorn,RMathias1990}{
      \BIB{#2}{R. Horn and R. Mathias}
         {Cauchy-Schwartz inequalities associated with positive semidefinite matrices}
         {\jRN{LAA}}{142}{1990}{63--82}{#1}}
   \ITEE{#3}{GEHuhunaisvili1955}{
      \BIB{#2}{G.E. Huhunai\v{s}vili}
         {On a property of Urysohn's universal metric space}
         {\jRN{DANSSSR}}{101}{1955}{607--610 (Russian)}{#1}}
   \ITEE{#3}{JEHumphreys1990}{
      \BIb{#2}{J.E. Humphreys}
         {Reflection Groups and Coxeter Groups}
         {Cambridge University Press}{1990}{#1}}
   \ITEE{#3}{JRIsbell1964}{
      \BIB{#2}{J.R. Isbell}
         {Six theorems about injective metric spaces}
         {\jRN{CMHelv}}{39}{1964}{65--76}{#1}}
   \ITEE{#3}{SIzumino,YKato1985}{
      \BIB{#2}{S. Izumino and Y. Kato}
         {The closure of invertible operators on Hilbert space}
         {\jRN{ActaSM}}{49}{1985}{321--327}{#1}}
   \ITEE{#3}{CJiang2004}{
      \BIB{#2}{C. Jiang}
         {Similarity classification of Cowen-Douglas operators}
         {\jRN{CanadJM}}{56}{2004}{742--775}{#1}}
   \ITEE{#3}{WBJohnson,JLindenstrauss2001}{
      \BiB{#2}{W.B. Johnson and J. Lindenstrauss}{Basic Concepts in the Geometry of Banach Spaces}
         {Chapter 1 in:}{Handbook of the Geometry of Banach Spaces, Vol. 1}
         {W.B. Johnson and J. Lindenstrauss (editors), Elsevier Science B.V., Amsterdam}{2001}{1--84}{#1}}
   \ITEE{#3}{IBJung,JStochel2008}{
      \BIB{#2}{I.B. Jung and J. Stochel}
         {Subnormal operators whose adjoints have rich point spectrum}
         {\jRN{JFA}}{255}{2008}{1797--1816}{#1}}
   \ITEE{#3}{RVKadison,JRRingrose1983}{
      \BIb{#2}{R.V. Kadison and J.R. Ringrose}
         {Fundamentals of the Theory of Operator Algebras. Volume I: Elementary Theory}
         {Academic Press, Inc., New York-London}{1983}{#1}}
   \ITEE{#3}{RVKadison,JRRingrose1986}{
      \BIb{#2}{R.V. Kadison and J.R. Ringrose}
         {Fundamentals of the Theory of Operator Algebras. Volume II: Advanced Theory}
         {Academic Press, Inc., Orlando-London}{1986}{#1}}
   \ITEE{#3}{SKakutani1936}{
      \BIB{#2}{S. Kakutani}
         {\"{U}ber die Metrisation der topologischen Gruppen}
         {\jRN{ProcImpAcadTokyo}}{12}{1936}{82--84}{#1}}
   \ITEE{#3}{SKakutani1938}{
      \BIB{#2}{S. Kakutani}
         {Two fixed-point theorems concerning bicompact convex sets}
         {\jRN{ProcImpAcadTokyo}}{14}{1938}{242--245}{#1}}
   \ITEE{#3}{SKakutani1941}{
      \BIB{#2}{S. Kakutani}
         {Concrete representation of abstract L-spaces}
         {\jRN{AnnM}}{42}{1941}{523--537}{#1}}
   \ITEE{#3}{SKakutani1941a}{
      \BIB{#2}{S. Kakutani}
         {Concrete representation of abstract M-spaces}
         {\jRN{AnnM}}{42}{1941}{994--1024}{#1}}
   \ITEE{#3}{NKalton2007}{
      \BIB{#2}{N. Kalton}
         {Extending Lipschitz maps into $\CCc(K)$-spaces}
         {\jRN{IsraelJM}}{162}{2007}{275--315}{#1}}
   \ITEE{#3}{RKane2001}{
      \BIb{#2}{R. Kane}
         {Reflection Groups and Invariant Theory}
         {Canadian Mathematical Society, Springer}{2001}{#1}}
   \ITEE{#3}{VKannan,SRRaju1980}{
      \BIB{#2}{V. Kannan and S.R. Raju}
         {The nonexistence of invariant universal measures on semigroups}
         {\jRN{PAMS}}{78}{1980}{482--484}{#1}}
   \ITEE{#3}{IKaplansky1951}{
      \BIB{#2}{I. Kaplansky}
         {A theorem on rings of operators}
         {\jRN{PacJM}}{1}{1951}{227--232}{#1}}
   \ITEE{#3}{MKatetov1988}{
      \BiB{#2}{M. Kat\v{e}tov}{On universal metric spaces}{in: Frolik (ed.),}
         {General Topology and its Relations to Modern Analysis and Algebra VI. Proceedings of the Sixth Prague 
         Topological Symposium 1986}{Heldermann Verlag Berlin}{1988}{323--330}{#1}}
   \ITEE{#3}{YKatznelson1960}{
      \BIB{#2}{Y. Katznelson}
         {Sur les alg\'{e}bres dont les \'{e}l\'{e}ments non n\'{e}gatifs admettent des racines carr\'{e}es}
         {\jRN{AnnSciEcNormSupT}}{77}{1960}{167--174}{#1}}
   \ITEE{#3}{OHKeller1931}{
      \BIB{#2}{O.H. Keller}
         {Die Homoiomorphie der kompakten konvexen Mengen in Hilbertschen Raum}
         {\jRN{MAnn}}{105}{1931}{748--758}{#1}}
   \ITEE{#3}{MAKhamsi,WAKirk,CMartinez2000}{
      \BIB{#2}{M.A. Khamsi, W.A. Kirk, C. Martinez}
         {Fixed point and selection theorems in hyperconvex spaces}
         {\jRN{PAMS}}{128}{2000}{3275--3283}{#1}}
   \ITEE{#3}{ABKhararazishvili1998}{
      \BIb{#2}{A.B. Khararazishvili}
         {Transformation groups and invariant measures. Set-theoretic aspects}
         {World Scientific Publishing Co., Inc., River Edge, NJ}{1998}{#1}}
   \ITEE{#3}{YKijima1987}{
      \BIB{#2}{Y. Kijima}
         {Fixed points of nonexpansive self-maps of a compact metric space}
         {\jRN{JMAnApp}}{123}{1987}{114--116}{#1}}
  \ITEE{#3}{JSKim,ChRKim,SGLee1980}{
      \BIB{#2}{J.S. Kim, Ch.R. Kim, S.G. Lee}
         {Reducing operator valued spectra of a Hilbert space operator}
         {\jRN{JKoreanMS}}{17}{1980}{123--129}{#1}}
   \ITEE{#3}{JKindler1995}{
      \BIB{#2}{J. Kindler}
         {Minimax theorems with applications to convex metric spaces}
         {\jRN{CollM}}{68}{1995}{179--186}{#1}}
   \ITEE{#3}{WAKirk1998}{
      \BIB{#2}{W.A. Kirk}
         {Hyperconvexity of $\RRR$-trees}
         {\jRN{FM}}{156}{1998}{67--72}{#1}}
   \ITEE{#3}{VLKleeJr1952}{
      \BIB{#2}{V.L. Klee Jr.}
         {Invariant metrics in groups (solution of a problem of Banach)}
         {\jRN{PAMS}}{3}{1952}{484--487}{#1}}
   \ITEE{#3}{HJKowalsky1957}{
      \BIB{#2}{H.J. Kowalsky}
         {Einbettung metrischer R\"{a}ume}
         {\jRN{ArchM}}{8}{1957}{336--339}{#1}}
   \ITEE{#3}{WKubis,MRubin2010}{
      \BIB{#2}{W. Kubi\'{s} and M. Rubin}
         {Extension and reconstruction theorems for the Urysohn universal metric space}
         {\jRN{CzMJ}}{60}{2010}{1--29}{#1}}
   \ITEE{#3}{KKuratowski1966}{
      \BIb{#2}{K. Kuratowski}
         {Topology. \textup{Vol. I}}
         {\jRN{PWN}}{1966}{#1}}
   \ITEE{#3}{KKuratowski,BKnaster1927}{
      \BIB{#2}{K. Kuratowski and B. Knaster}
         {A connected and connected im kleinen point set which contains no perfect subset}
         {\jRN{BAMS}}{33}{1927}{106--109}{#1}}
   \ITEE{#3}{KKuratowski,AMostowski1976}{
      \BIb{#2}{K. Kuratowski and A. Mostowski}
         {Set Theory with an Introduction to Descriptive Set Theory}
         {\jRN{PWN}}{1976}{#1}}
   \ITEE{#3}{GLewicki1992}{
      \BIB{#2}{G. Lewicki}
         {Bernstein's ``lethargy'' theorem in metrizable topological linear spaces}
         {\jRN{MonatM}}{113}{1992}{213--226}{#1}}
   \ITEE{#3}{ASLewis1996}{
      \BIB{#2}{A.S. Lewis}
         {Group invariance and convex matrix analysis}
         {\jRN{SIAMJMAA}}{17}{1996}{927--949}{#1}}
   \ITEE{#3}{C-KLi,N-KTsing1991}{
      \BIB{#2}{C.-K. Li and N.-K. Tsing}
         {$G$-invariant norms and $G(c)$-radii}
         {\jRN{LAA}}{150}{1991}{179--194}{#1}}
   \ITEE{#3}{AJLazar,JLindenstrauss1971}{
      \BIB{#2}{A.J. Lazar and J. Lindenstrauss}
         {Banach spaces whose duals are $L_1$ spaces and their representing matrices}
         {\jRN{ActaM}}{126}{1971}{165--193}{#1}}
   \ITEE{#3}{EHLieb,MLoss1997}{
      \BIb{#2}{E.H. Lieb and M. Loss}
         {Analysis \textup{(Graduate Studies in Mathematics, vol. 14)}}
         {Amer. Math. Soc., Providence, RI}{1997}{#1}}
   \ITEE{#3}{ALindenbaum1926}{
      \BIB{#2}{A. Lindenbaum}
         {Contributions \`{a} l'\'{e}tude de l'espace m\'{e}trique I}
         {\jRN{FM}}{8}{1926}{209--222}{#1}}
   \ITEE{#3}{DLindenstrauss,LTzafriri1971}{
      \BIB{#2}{D. Lindenstrauss and L. Tzafriri}
         {On the complemented subspaces problem}
         {\jRN{IsraelJM}}{9}{1971}{263--269}{#1}}
   \ITEE{#3}{RILoebl1986}{
      \BIB{#2}{R.I. Loebl}
         {A note on containment of operators}
         {\jRN{BAustrMS}}{33}{1986}{279--291}{#1}}
   \ITEE{#3}{LHLoomis1945}{
      \BIB{#2}{L.H. Loomis}
         {Abstract congruence and the uniqueness of Haar measure}
         {\jRN{AnnM}}{46}{1945}{348--355}{#1}}
   \ITEE{#3}{LHLoomis1949}{
      \BIB{#2}{L.H. Loomis}
         {Haar measure in uniform structures}
         {\jRN{DukeMJ}}{16}{1949}{193--208}{#1}}
   \ITEE{#3}{ERLorch1939}{
      \BIB{#2}{E.R. Lorch}
         {Bicontinuous linear transformation in certain vector spaces}
         {\jRN{BAMS}}{45}{1939}{564--569}{#1}}
   \ITEE{#3}{ATLundell,SWeingram1969}{
      \BIb{#2}{A.T. Lundell and S. Weingram}
         {The topology of CW-complexes}
         {Litton Educ. Publ.}{1969}{#1}}
   \ITEE{#3}{WLusky1976}{
      \BIB{#2}{W. Lusky}
         {The Gurarij spaces are unique}
         {\jRN{ArchM}}{27}{1976}{627--635}{#1}}
   \ITEE{#3}{WLusky1977}{
      \BIB{#2}{W. Lusky}
         {On separable Lindenstrauss spaces}
         {\jRN{JFA}}{26}{1977}{103--120}{#1}}
   \ITEE{#3}{DMaharam1942}{
      \BIB{#2}{D. Maharam}
         {On homogeneous measure algebras}
         {\jRN{PNatlUSA}}{28}{1942}{108--111}{#1}}
   \ITEE{#3}{MMalicki,SSolecki2009}{
      \BIB{#2}{M. Malicki and S. Solecki}
         {Isometry groups of separable metric spaces}
         {\jRN{MProcCambPhS}}{146}{2009}{67--81}{#1}}
   \ITEE{#3}{PMankiewicz1972}{
      \BIB{#2}{P. Mankiewicz}
         {On extension of isometries in normed linear spaces}
         {\jRN{BAPolSSSM}}{20}{1972}{367--371}{#1}}
   \ITEE{#3}{JMartinezMaurica,MTPellon1987}{
      \BIB{#2}{J. Martinez-Maurica and M.T. Pell\'{o}n}
         {Non-archimedean Chebyshev centers}
         {\jRN{IndagMP}}{90}{1987}{417--421}{#1}}
   \ITEE{#3}{KMaurin1980}{
      \BIb{#2}{K. Maurin}
         {Analysis, Part II}
         {D. Reidel, Dordrecht-Boston-London}{1980}{#1}}
   \ITEE{#3}{SMazur,SUlam1932}{
      \BIB{#2}{S. Mazur and S. Ulam}
         {Sur les transformationes isom\'{e}triques d'espaces vectoriels norm\'{e}s}
         {\jRN{CRASParis}}{194}{1932}{946--948}{#1}}
   \ITEE{#3}{SMazurkiewicz1920}{
      \BIB{#2}{S. Mazurkiewicz}
         {Sur les lignes de Jordan}
         {\jRN{FM}}{1}{1920}{166--209}{#1}}
   \ITEE{#3}{SMazurkiewicz,WSierpinski1920}{
      \BIB{#2}{S. Mazurkiewicz and W. Sierpi\'{n}ski}
         {Contributions a la topologie des ensembles denombrables}
         {\jRN{FM}}{1}{1920}{17--27}{#1}}
   \ITEE{#3}{MMbekhta1992}{
      \BIB{#2}{M. Mbekhta}
         {Sur la structure des composantes connexes semi-Fredholm de $B(H)$}
         {\jRN{PAMS}}{116}{1992}{521--524}{#1}}
   \ITEE{#3}{JEMcCarthy1996}{
      \BIB{#2}{J.E. McCarthy}
         {Boundary values and Cowen-Douglas curvature}
         {\jRN{JFA}}{137}{1996}{1--18}{#1}}
   \ITEE{#3}{JMelleray2007}{
      \BIB{#2}{J. Melleray}
         {Computing the complexity of the relation of isometry between separable Banach spaces}
         {\jRN{MLQ}}{53}{2007}{128--131}{#1}}
   \ITEE{#3}{JMelleray2007a}{
      \BIB{#2}{J. Melleray}
         {On the geometry of Urysohn's universal metric space}
         {\jRN{TopA}}{154}{2007}{384--403}{#1}}
   \ITEE{#3}{JMelleray2008}{
      \BIB{#2}{J. Melleray}
         {Some geometric and dynamical properties of the Urysohn space}
         {\jRN{TopA}}{155}{2008}{1531--1560}{#1}}
   \ITEE{#3}{JMelleray,FVPetrov,AMVershik2008}{
      \BIB{#2}{J. Melleray, F.V. Petrov, A.M. Vershik}
         {Linearly rigid metric spaces and the embedding problem}
         {\jRN{FM}}{199}{2008}{177--194}{#1}}
   \ITEE{#3}{EMichael1953}{
      \BIB{#2}{E. Michael}
         {Some extension theorems for continuous functions}
         {\jRN{PacJM}}{3}{1953}{789--806}{#1}}
   \ITEE{#3}{EMichael1954}{
      \BIB{#2}{E. Michael}
         {Local properties of topological spaces}
         {\jRN{DukeMJ}}{21}{1954}{163--171}{#1}}
   \ITEE{#3}{EMichael1956}{
      \BIB{#2}{E. Michael}
         {Selected selection theorems}
         {\jRN{AmMMon}}{58}{1956}{233--238}{#1}}
   \ITEE{#3}{EMichael1956a}{
      \BIB{#2}{E. Michael}
         {Continuous selections. I}
         {\jRN{AnnM}}{63}{1956}{361--382}{#1}}
   \ITEE{#3}{EMichael1956b}{
      \BIB{#2}{E. Michael}
         {Continuous selections. II}
         {\jRN{AnnM}}{64}{1956}{562--580}{#1}}
   \ITEE{#3}{EMichael1959}{
      \BIB{#2}{E. Michael}
         {A theorem on semi-continuous set-valued functions}
         {\jRN{DukeMJ}}{26}{1959}{647--652}{#1}}
   \ITEE{#3}{JVanMill1986}{
      \BIB{#2}{J. van Mill}
         {Another counterexample in ANR theory}
         {\jRN{PAMS}}{97}{1986}{136--138}{#1}}
   \ITEE{#3}{JVanMill2001}{
      \BIb{#2}{J. van Mill}
         {The Infinite-Dimensional Topology of Function Spaces 
         \textup{(North-Holland Mathematical Library, vol. 64)}}
         {Elsevier, Amsterdam}{2001}{#1}}
   \ITEE{#3}{WMlak1991}{
      \BIb{#2}{W. Mlak}
         {Hilbert Spaces and Operator Theory}
         {PWN --- Polish Scientific Publishers and Kluwer Academic Publishers, Warszawa-Dordrecht}{1991}{#1}}
   \ITEE{#3}{JMogilski1979}{
      \BIB{#2}{J. Mogilski}
         {$CE$-decomposition of $l_2$-manifolds}
         {\jRN{BAPolSSSM}}{27}{1979}{309--314}{#1}}
   \ITEE{#3}{RLMoore1916}{
      \BIB{#2}{R.L. Moore}
         {On the foundations of plane analysis situs}
         {\jRN{TAMS}}{17}{1916}{131--164}{#1}}
   \ITEE{#3}{KMorita1955}{
      \BIB{#2}{K. Morita}
         {A condition for the metrizability of topological spaces and for $n$-dimensionality}
         {\jRN{SciRepTokyoA}}{5}{1955}{33--36}{#1}}
   \ITEE{#3}{AMukherjea,NATserpes1976}{
      \BIb{#2}{A. Mukherjea and N.A. Tserpes}
         {Measures on topological semigroups}
         {Springer Lecture Notes in Math. Vol. 547, Berlin}{1976}{#1}}
   \ITEE{#3}{JMycielski1974}{
      \BIB{#2}{J. Mycielski}
         {Remarks on invariant measures in metric spaces}
         {\jRN{CollM}}{32}{1974}{105--112}{#1}}
   \ITEE{#3}{SNNaboko1984}{
      \BIB{#2}{S.N. Naboko}
         {Conditions for similarity to unitary and selfadjoint operators}
         {\jRN{FunkAnalPril}}{18}{1984}{16--27}{#1}}
   \ITEE{#3}{LNachbin1965}{
      \BIb{#2}{L. Nachbin}
         {The Haar Integral}
         {D. Van Nostrand Company, Inc., Princeton-New Jersey-Toronto-New York-London}{1965}{#1}}
   \ITEE{#3}{TDNarang,SKGarg1991}{
      \BIB{#2}{T.D. Narang and S.K. Garg}
         {On the uniqueness of best approximation in non-archimedian spaces}
         {\jRN{PeriodMHung}}{22}{1991}{121--124}{#1}}
   \ITEE{#3}{JVonNeumann1930}{
      \BIB{#2}{J. von Neumann}
         {Zur Algebra der Funktionaloperationen und Theorie der normalen Operatoren}
         {\jRN{MAnn}}{102}{1930}{370--427}{#1}}
   \ITEE{#3}{JVonNeumann1934}{
      \BIB{#2}{J. von Neumann}
         {Zum Haarschen Mass in topologischen Gruppen}
         {\jRN{ComposM}}{1}{1934}{106--114}{#1}}
   \ITEE{#3}{JVonNeumann1937}{
      \BiB{#2}{J. von Neumann}
         {Some matrix-inequalities and metrization of matrix-space}{\jRN{TomskUnivRev}{} \textbf{1} (1937), 286--300; 
         in }{Collected Works}{Pergamon, New York}{1962}{Vol. 4, 205--219}{#1}}
   \ITEE{#3}{JVonNeumann1949}{
      \BIB{#2}{J. von Neumann}
         {On Rings of Operators. Reduction Theory}
         {\jRN{AnnM}}{50}{1949}{401--485}{#1}}
   \ITEE{#3}{ONielson1973}{
      \BIB{#2}{O. Nielson}
         {Borel sets of von Neumann algebras}
         {\jRN{AmJM}}{95}{1973}{145--164}{#1}}
   \ITEE{#3}{pn2002}{\bibITEM{#2}{#1} \mypaplist{pn1}}
   \ITEE{#3}{pn2006a}{\bibITEM{#2}{#1} \mypaplist{pn2}}
   \ITEE{#3}{pn2006b}{\bibITEM{#2}{#1} \mypaplist{pn3}}
   \ITEE{#3}{pn2007}{\bibITEM{#2}{#1} \mypaplist{pn4}}
   \ITEE{#3}{pn2008a}{\bibITEM{#2}{#1} \mypaplist{pn5}}
   \ITEE{#3}{pn2008b}{\bibITEM{#2}{#1} \mypaplist{pn6}}
   \ITEE{#3}{pn2009a}{\bibITEM{#2}{#1} \mypaplist{pn7}}
   \ITEE{#3}{pn2009b}{\bibITEM{#2}{#1} \mypaplist{pn8}}
   \ITEE{#3}{pn2009c}{\bibITEM{#2}{#1} \mypaplist{pn9}}
   \ITEE{#3}{pn2010a}{\bibITEM{#2}{#1} \mypaplist{pn12}}
   \ITEE{#3}{pn2010b}{\bibITEM{#2}{#1} \mypaplist{pn13}}
   \ITEE{#3}{pn2011a}{\bibITEM{#2}{#1} \mypaplist{pn10}}
   \ITEE{#3}{pn2011b}{\bibITEM{#2}{#1} \mypaplist{pn15}}
   \ITEE{#3}{pn2011c}{\bibITEM{#2}{#1} \mypaplist{pn16}}
   \ITEE{#3}{pn2011d}{\bibITEM{#2}{#1} \mypaplist{pn17}}
   \ITEE{#3}{pn2009x}{
      \bibITEM{#2}{#1} \mypaplist{pn11}}
   \ITEE{#3}{pn2010x}{
      \bibITEM{#2}{#1} \mypaplist{pn14}}
   \ITEE{#3}{pnXXXXb}{
      \bibITEM{#2}{#1} \mypaplist{pnX2}}
   \ITEE{#3}{pnXXXXc}{
      \bibITEM{#2}{#1} \mypaplist{pnX3}}
   \ITEE{#3}{pnXXXXd}{
      \bibITEM{#2}{#1} \mypaplist{pnX13}}
   \ITEE{#3}{MNiezgoda1998}{
      \BIB{#2}{M. Niezgoda}
         {Group majorization and Schur type inequalities}
         {\jRN{LAA}}{268}{1998}{9--30}{#1}}
   \ITEE{#3}{MNiezgoda1998a}{
      \BIB{#2}{M. Niezgoda}
         {An analytical characterization of effective and of irreducible groups inducing cone orderings}
         {\jRN{LAA}}{269}{1998}{105--114}{#1}}
   \ITEE{#3}{MNiezgoda,TYTam2001}{
      \BIB{#2}{M. Niezgoda and T.Y. Tam}
         {On norm property of $G(c)$-radii and Eaton triples}
         {\jRN{LAA}}{336}{2001}{119--130}{#1}}
   \ITEE{#3}{APazy1983}{
      \BIb{#2}{A. Pazy}{Semigroups of Linear Operators 
         and Applications to Partial Differential Equations \textup{(Applied Mathematical Sciences, vol. 44)}}
         {Springer-Verlag, New York}{1983}{#1}}
   \ITEE{#3}{APelc1982}{
      \BIB{#2}{A. Pelc}
         {Semiregular invariant measures on abelian groups}
         {\jRN{PAMS}}{86}{1982}{423--426}{#1}}
   \ITEE{#3}{RPenrose1955}{
      \BIB{#2}{R. Penrose}
         {A generalized inverse for matrices}
         {\jRN{ProcCambPhS}}{51}{1955}{406--413}{#1}}
   \ITEE{#3}{VPestov2006}{
      \BIb{#2}{V. Pestov}
         {Dynamics of infinite-dimensional groups. The Ramsey-Dvoretzky-Milman phenomenon}
         {University Lecture Series \textbf{40}, AMS, Providence, RI}{2006}{#1}}
   \ITEE{#3}{VPestov2007}{
      \BiB{#2}{V. Pestov}
         {Forty-plus annotated questions about large topological groups}
         {in:}{Open Problems in Topology II}{Elliot Pearl (editor), Elsevier B.V., Amsterdam}{2007}{439--450}{#1}}
   \ITEE{#3}{PVPetersen1993}{
      \BiB{#2}{P.V. Petersen}
         {Gromov-Hausdorff convergence of metric spaces}{in book:}{Differential Geometry: Riemannian Geometry 
         (Los Angeles, CA, 1990)}{Amer. Math. Soc., Providence, RI}{1993}{489--504}{#1}}
   \ITEE{#3}{DRamachandran,MMisiurewicz1982}{
      \BIB{#2}{D. Ramachandran and M. Misiurewicz}
         {Hopf's theorem on invariant measures for a group of transformations}
         {\jRN{SM}}{74}{1982}{183--189}{#1}}
   \ITEE{#3}{JMRosenblatt1974}{
      \BIB{#2}{J.M. Rosenblatt}
         {Equivalent invariant measures}
         {\jRN{IsraelJM}}{17}{1974}{261--270}{#1}}
   \ITEE{#3}{HLRoyden1963}{
      \BIb{#2}{H.L. Royden}
         {Real Analysis}
         {The Macmillan Co., New York}{1963}{#1}}
   \ITEE{#3}{WRudin1962}{
      \BIb{#2}{W. Rudin}
         {Fourier Analysis on Groups \textup{(Interscience Tracts in Pure and Applied Mathematics, Number 12)}}
         {Interscience Publishers, New York}{1962}{#1}}
   \ITEE{#3}{WRudin1991}{
      \BIb{#2}{W. Rudin}
         {Functional Analysis}
         {McGraw-Hill Science}{1991}{#1}}
   \ITEE{#3}{TSaito1972}{
      \BiB{#2}{T. Sait\^{o}}{Generations of von Neumann algebras}
         {Lecture Notes in Math. vol. 247}{\textup{(}Lecture on Operator Algebras\textup{)}}
         {Springer, Berlin-Heidelberg-New York}{1972}{435--531}{#1}}
   \ITEE{#3}{KSakai,MYaguchi2003}{
      \BIB{#2}{K. Sakai and M. Yaguchi}
         {Characterizing manifolds modeled on certain dense subspaces of non-separable Hilbert spaces}
         {\jRN{TsukubaJM}}{27}{2003}{143--159}{#1}}
   \ITEE{#3}{SSakai1971}{
      \BIb{#2}{S. Sakai}
         {$\CCc^*$-Algebras and $\WWw^*$-Algebras}
         {Springer-Verlag, Berlin-Heidelberg-New York}{1971}{#1}}
   \ITEE{#3}{RSchori1971}{
      \BIB{#2}{R. Schori}
         {Topological stability for infinite-dimensional manifolds}
         {\jRN{ComposM}}{23}{1971}{87--100}{#1}}
   \ITEE{#3}{JTSchwartz1967}{
      \BIb{#2}{J.T. Schwartz}
         {$\WWw^*$-algebras}
         {Gordon and Breach, Science Publishers Inc., New York-London-Paris}{1967}{#1}}
   \ITEE{#3}{ZSemadeni1971}{
      \BIb{#2}{Z. Semadeni}
         {Banach Spaces of Continuous Functions (Vol. I)}
         {\jRN{PWN}}{1971}{#1}}
   \ITEE{#3}{JPSerre1951}{
      \BIB{#2}{J.-P. Serre}
         {Homologie singuli\`{e}re des espaces fibr\'{e}s}
         {\jRN{AnnM}}{54}{1951}{425--505}{#1}}
   \ITEE{#3}{DSherman2007}{
      \BIB{#2}{D. Sherman}
         {On the dimension theory of von Neumann algebras}
         {\jRN{MScand}}{101}{2007}{123--147}{#1}}
   \ITEE{#3}{WSierpinski1928}{
      \BIB{#2}{W. Sierpi\'{n}ski}
         {Sur les projections des ensembles compl\'{e}mentaires aux ensembles \textup{(A)}}
         {\jRN{FM}}{11}{1928}{117--122}{#1}}
   \ITEE{#3}{MSlocinski1980}{
      \BIB{#2}{M. S\l{}oci\'{n}ski}
         {On the Wold-type decomposition of a pair of commuting isometries}
         {\jRN{APM}}{37}{1980}{255--262}{#1}}
   \ITEE{#3}{RCSteinlage1975}{
      \BIB{#2}{R.C. Steinlage}
         {On Haar measure in locally compact $T_2$ spaces}
         {\jRN{AmJM}}{97}{1975}{291--307}{#1}}
   \ITEE{#3}{JStochel,FHSzafraniec1989}{
      \BIB{#2}{J. Stochel and F.H. Szafraniec}
         {On normal extensions of unbounded operators. III. Spectral properties}
         {\jRN{PublRIMSKyoto}}{25}{1989}{105--139}{#1}}
   \ITEE{#3}{JStochel,FHSzafraniec1989a}{
      \BIB{#2}{J. Stochel and F.H. Szafraniec}
         {The normal part of an unbounded operator}
         {\jRN{ProcKonink}}{92}{1989}{495--503}{#1}}
   \ITEE{#3}{AHStone1962}{
      \BIB{#2}{A.H. Stone}
         {Absolute $\FFf_{\sigma}$-spaces}
         {\jRN{PAMS}}{13}{1962}{495--499}{#1}}
   \ITEE{#3}{AHStone1962a}{
      \BIB{#2}{A.H. Stone}
         {Non-separable Borel sets}
         {\jRN{DissM}}{28}{1962}{41 pages}{#1}}
   \ITEE{#3}{AHStone1972}{
      \BIB{#2}{A.H. Stone}
         {Non-separable Borel sets II}
         {\jRN{GTopA}}{2}{1972}{249--270}{#1}}
   \ITEE{#3}{MHStone1937}{
      \BIB{#2}{M.H. Stone}
         {Application of the theory of Boolean rings to general topology}
         {\jRN{TAMS}}{41}{1937}{375--481}{#1}}
   \ITEE{#3}{MHStone1948}{
      \BIB{#2}{M.H. Stone}
         {The generalized Weierstrass approximation theorem}
         {\jRN{MMag}}{21}{1948}{167--184}{#1}}
   \ITEE{#3}{BSz-Nagy1947}{
      \BIB{#2}{B. Sz.-Nagy}
         {On uniformly bounded linear transformations in Hilbert space}
         {\jRN{ActaSM}}{11}{1947}{152--157}{#1}}
   \ITEE{#3}{WTakahashi1970}{
      \BIB{#2}{W. Takahashi}
         {A convexity in metric space and nonexpansive mappings, I}
         {\jRN{KodaiMSemRep}}{22}{1970}{142--149}{#1}}
   \ITEE{#3}{MTakesaki2002}{
      \BIb{#2}{M. Takesaki}
         {Theory of Operator Algebras I \textup{(Encyclopaedia of Mathematical Sciences, Volume 124)}}
         {Springer-Verlag, Berlin-Heidelberg-New York}{2002}{#1}}
   \ITEE{#3}{MTakesaki2003}{
      \BIb{#2}{M. Takesaki}
         {Theory of Operator Algebras II \textup{(Encyclopaedia of Mathematical Sciences, Volume 125)}}
         {Springer-Verlag, Berlin-Heidelberg-New York}{2003}{#1}}
   \ITEE{#3}{MTakesaki2003a}{
      \BIb{#2}{M. Takesaki}
         {Theory of Operator Algebras III \textup{(Encyclopaedia of Mathematical Sciences, Volume 127)}}
         {Springer-Verlag, Berlin-Heidelberg-New York}{2003}{#1}}
   \ITEE{#3}{TYTam1999}{
      \BIB{#2}{T.Y. Tam}
         {An extension of a result of Lewis}
         {\jRN{ELA}}{5}{1999}{1--10}{#1}}
   \ITEE{#3}{TYTam2000}{
      \BIB{#2}{T.Y. Tam}
         {Group majorization, Eaton triples and numerical range}
         {\jRN{LMLA}}{47}{2000}{11--28}{#1}}
   \ITEE{#3}{TYTam2002}{
      \BIB{#2}{T.Y. Tam}
         {Generalized Schur-concave functions and Eaton triples}
         {\jRN{LMLA}}{50}{2002}{113--120}{#1}}
   \ITEE{#3}{TYTam,WCHill2001}{
      \BIB{#2}{T.Y. Tam and W.C. Hill}
         {On $G$-invariant norms}
         {\jRN{LAA}}{331}{2001}{101--112}{#1}}
   \ITEE{#3}{AFTiman,IAVestfrid1983}{
      \BIB{#2}{A.F. Timan and I.A. Vestfrid}
         {Any separable ultrametric space can be isometrically imbedded in $l_2$}
         {\jRN{FAA}}{17}{1983}{70--71}{#1}}
   \ITEE{#3}{JTomiyama1958}{
      \BIB{#2}{J. Tomiyama}
         {Generalized dimension function for $\WWw^*$-algebras of infinite type}
         {\jRN{TohokuMJ} (2)}{10}{1958}{121--129}{#1}}
   \ITEE{#3}{HTorunczyk1970}{
      \BIB{#2}{H. Toru\'{n}czyk}
         {Remarks on Anderson's paper ``On topological infinite deficiency''}
         {\jRN{FM}}{66}{1970}{393--401}{#1}}
   \ITEE{#3}{HTorunczyk1970a}{
      \BIb{#2}{H. Toru\'{n}czyk}
         {$G$-$K$-absorbing and skeletonized sets in metric spaces}
         {Ph.D. thesis, Inst. Math. Polish Acad. Sci., Warszawa}{1970}{#1}}
   \ITEE{#3}{HTorunczyk1972}{
      \BIB{#2}{H. Toru\'{n}czyk}
         {A short proof of Hausdorff's theorem on extending metrics}
         {\jRN{FM}}{77}{1972}{191--193}{#1}}
   \ITEE{#3}{HTorunczyk1974}{
      \BIB{#2}{H. Toru\'{n}czyk}
         {Absolute retracts as factors of normed linear spaces}
         {\jRN{FM}}{86}{1974}{53--67}{#1}}
   \ITEE{#3}{HTorunczyk1975}{
      \BIB{#2}{H. Toru\'{n}czyk}
         {On Cartesian factors and the topological classification of linear metric spaces}
         {\jRN{FM}}{88}{1975}{71--86}{#1}}
   \ITEE{#3}{HTorunczyk1978}{
      \BIB{#2}{H. Toru\'{n}czyk}
         {Concerning locally homotopy negligible sets and characterization of $l_2$-manifolds}
         {\jRN{FM}}{101}{1978}{93--110}{#1}}
   \ITEE{#3}{HTorunczyk1980}{
      \BiB{#2}{H. Toru\'{n}czyk}{Characterization of infinite-dimensional manifolds}{in:}
         {Proceedings of the International Conference on Geometric Topology (Warsaw, 1978)}
         {\jRN{PWN}}{1980}{431--437}{#1}}
   \ITEE{#3}{HTorunczyk1981}{
      \BIB{#2}{H. Toru\'{n}czyk}
         {Characterizing Hilbert space topology}
         {\jRN{FM}}{111}{1981}{247--262}{#1}}
   \ITEE{#3}{HTorunczyk1985}{
      \BIB{#2}{H. Toru\'{n}czyk}
         {A correction of two papers concerning Hilbert manifolds}
         {\jRN{FM}}{125}{1985}{89--93}{#1}}
   \ITEE{#3}{KTsuda1985}{
      \BIB{#2}{K. Tsuda}
         {A note on closed embeddings of finite dimensional metric spaces}
         {\jRN{BLondMS}}{17}{1985}{273--278}{#1}}
   \ITEE{#3}{PSUrysohn1925}{
      \BIB{#2}{P.S. Urysohn}
         {Sur un espace m\'{e}trique universel}
         {\jRN{CRASParis}}{180}{1925}{803--806}{#1}}
   \ITEE{#3}{PSUrysohn1927}{
      \BIB{#2}{P.S. Urysohn}
         {Sur un espace m\'{e}trique universel}
         {\jRN{BullSM}}{51}{1927}{43--64, 74--96}{#1}}
   \ITEE{#3}{VVUspenskij1986}{
      \BIB{#2}{V.V. Uspenskij}
         {A universal topological group with a countable basis}
         {\jRN{FAA}}{20}{1986}{86--87}{#1}}
   \ITEE{#3}{VVUspenskij1990}{
      \BIB{#2}{V.V. Uspenskij}
         {On the group of isometries of the Urysohn universal metric space}
         {\jRN{CMUC}}{31}{1990}{181--182}{#1}}
   \ITEE{#3}{VVUspenskij2004}{
      \BIB{#2}{V.V. Uspenskij}
         {The Urysohn universal metric space is homeomorphic to a Hilbert space}
         {\jRN{TopA}}{139}{2004}{145--149}{#1}}
   \ITEE{#3}{VVUspenskij2008}{
      \BIB{#2}{V.V. Uspenskij}
         {On subgroups of minimal topological groups}
         {\jRN{TopA}}{155}{2008}{1580--1606}{#1}}
   \ITEE{#3}{VSVaradarajan1963}{
      \BIB{#2}{V.S. Varadarajan}
         {Groups of automorphisms of Borel spaces}
         {\jRN{TAMS}}{109}{1963}{191--220}{#1}}
   \ITEE{#3}{AMVershik1998}{
      \BIB{#2}{A.M. Vershik}
         {The universal Urysohn space, Gromov's metric triples, and random metrics on the series of natural numbers}
         {\jRN{UspekhiMN}}{53}{1998}{57--64}{#1} English translation: \jRN{RussMS}{} \textbf{53} (1998), 921--928. 
         Correction: \jRN{UspekhiMN}{} \textbf{56} (2001), p. 207. English translation: \jRN{RussMS}{} \textbf{56} 
         (2001), p. 1015.}
   \ITEE{#3}{AMVershik2002}{
      \BIb{#2}{A.M. Vershik}
         {Random metric spaces and the universal Urysohn space}
         {Fundamental Mathematics Today. 10th anniversary of the Independent Moscow University. MCCME Publ.}{2002}{#1}}
   \ITEE{#3}{NWeaver1999}{
      \BIb{#2}{N. Weaver}
         {Lipschitz Algebras}
         {World Scientific}{1999}{#1}}
   \ITEE{#3}{JWeidmann1980}{
      \BIb{#2}{J. Weidmann}
         {Linear Operators in Hilbert Spaces}
         {(Graduate Texts in Mathematics, vol. 68) Springer-Verlag New York Inc.}{1980}{#1}}
   \ITEE{#3}{JEWest1969}{
      \BIB{#2}{J.E. West}
         {Approximating homotopies by isotopies in Fr\'{e}chet manifolds}
         {\jRN{BAMS}}{75}{1969}{1254--1257}{#1}}
   \ITEE{#3}{JEWest1969a}{
      \BIB{#2}{J.E. West}
         {Fixed-point sets of transformation groups on infinite-product spaces}
         {\jRN{PAMS}}{21}{1969}{575--582}{#1}}
   \ITEE{#3}{JEWest1970}{
      \BIB{#2}{J.E. West}
         {The ambient homeomorphy of infinite-dimensional Hilbert spaces}
         {\jRN{PacJM}}{34}{1970}{257--267}{#1}}
   \ITEE{#3}{JHCWhitehead1949}{
      \BIB{#2}{J.H.C. Whitehead}
         {Combinatorial homotopy I}
         {\jRN{BAMS}}{55}{1949}{213--245}{#1}}
   \ITEE{#3}{GTWhyburn1942}{
      \BIb{#2}{G. T. Whyburn}
         {Analytic Topology}
         {Amer. Math. Soc. Colloquium Publications (vol. XXVIII), New York}{1942}{#1}}
   \ITEE{#3}{WWogen1969}{
      \BIB{#2}{W. Wogen}
         {On generators for von Neumann algebras}
         {\jRN{BAMS}}{75}{1969}{95--99}{#1}}
   \ITEE{#3}{RYTWong1967}{
      \BIB{#2}{R.Y.T. Wong}
         {On homeomorphisms of certain infinite dimensional spaces}
         {\jRN{TAMS}}{128}{1967}{148--154}{#1}}
   \ITEE{#3}{LYang,JZhang1987}{
      \BIB{#2}{L. Yang and J. Zhang}
         {Average distance constants of some compact convex space}
         {\jRN{JChinUST}}{17}{1987}{17--23}{#1}}
   \ITEE{#3}{PZakrzewski1993}{
      \BIB{#2}{P. Zakrzewski}
         {The existence of invariant $\sigma$-finite measures for a group of transformations}
         {\jRN{IsraelJM}}{83}{1993}{275--287}{#1}}
   \ITEE{#3}{PZakrzewski2002}{
      \BIb{#2}{P. Zakrzewski}
         {Measures on Algebraic-Topological Structures, Handbook of Measure Thoery}
         {E. Pap, ed., Elsevier, Amsterdam}{2002, 1091--1130}{#1}}
   \ITEE{#3}{KZhu2000}{
      \BIB{#2}{K. Zhu}
         {Operators in Cowen-Douglas classes}
         {\jRN{IllinoisJM}}{44}{2000}{767--783}{#1}}
   }

\newcommand{\mypaplist}[2][]{
   \ITEE{#2}{pn1}{
      \myBIB{Separate and joint similarity to families of normal operators}
         {\jRN[#1]{SM}}{149}{2002}{39--62}}
   \ITEE{#2}{pn2}{
      \myBIB{Locally arcwise connected metrizable spaces with the fixed point property are complete-metrizable}
         {\jRN[#1]{TopA}}{153}{2006}{1639--1642}}
   \ITEE{#2}{pn3}{
      \myBIB{Invariant measures for equicontinuous semigroups of continuous transformations of a compact Hausdorff space}
         {\jRN[#1]{TopA}}{153}{2006}{3373--3382}}
   \ITEE{#2}{pn4}{
      \myBIB{Approximation of the Hausdorff distance by the distance of continuous surjections}
         {\jRN[#1]{TopA}}{154}{2007}{655--664}}
   \ITEE{#2}{pn5}{
      \myBIB{Generalized Haar integral}
         {\jRN[#1]{TopA}}{155}{2008}{1323--1328}}
   \ITEE{#2}{pn6}{
      \myBIB{Integration and Lipschitz functions}
         {\jRN[#1]{RCMP}}{57}{2008}{391--399}}
   \ITEE{#2}{pn7}{
      \myBIB{Canonical Banach function spaces generated by Urysohn universal spaces. Measures as Lipschitz maps}
         {\jRN[#1]{SM}}{192}{2009}{97--110}}
   \ITEE{#2}{pn8}{
      \myBIB{Urysohn universal spaces as metric groups of exponent $2$}
         {\jRN[#1]{FM}}{204}{2009}{1--6}}
   \ITEE{#2}{pn9}{
      \myBIB{Central subsets of Urysohn universal spaces}
         {\jRN[#1]{CMUC}}{50}{2009}{445--461}}
   \ITEE{#2}{pn10}{
      \myBIB[P. Niemiec and T.Y. Tam]{A representation of $G$-in\-variant norms for Eaton triple}
         {\jRN[#1]{JCA}}{18}{2011}{59--65}}
   \ITEE{#2}{pn11}{
      \myBIB{Functor of extension of contractions on Urysohn universal spaces}
         {\jRN[#1]{ACS}}{}{2009}{\texttt{DOI: 10.1007/s10485-009-9218-z}}}
   \ITEE{#2}{pn12}{
      \myBIB{Ultra-$\mM$-separability}
         {\jRN[#1]{TopA}}{157}{2010}{669--673}}
   \ITEE{#2}{pn13}{
      \myBIB{Functor of extension of $\Lambda$-isometric maps between central subsets 
         of the unbounded Urysohn universal space}{\jRN[#1]{CMUC}}{51}{2010}{541--549}}
   \ITEE{#2}{pn14}{
      \myBIB{Normed topological pseudovector groups}{\jRN[#1]{ACS}}{}{2010}
         {\ITE{\equal{#1}{}}{\texttt{DOI: 10.1007/s10485\-010-9239-7}}{\texttt{DOI: 10.1007/s10485-010-9239-7}}}}
   \ITEE{#2}{pn15}{
      \myBIB{Topological structure of Urysohn universal spaces}
         {\jRN[#1]{TopA}}{158}{2011}{352--359}}
   \ITEE{#2}{pn16}{
      \myBIB{A note on invariant measures}
         {\jRN[#1]{OpusM}}{31}{2011}{425--431}}
   \ITEE{#2}{pn17}{
      \myBIB{Strengthened Stone-Weierstrass type theorem}
         {\jRN[#1]{OpusM}}{31}{2011}{645--650}}
   \ITEE{#2}{pnX2}{
      \myBAPP{Functor of continuation in Hilbert cube and Hilbert space}
         {to appear in \jRN[#1]{FM}}}
   \ITEE{#2}{pnX3}{
      \myBAPP{Norm closures of orbits of bounded operators}
         {to appear.}}
   \ITEE{#2}{pnX6}{
      \myBAPP{Extending maps by injective $\sigma$-$Z$-maps in Hilbert manifolds}
         {to appear in \jRN[#1]{BullPol}}}
   \ITEE{#2}{pnX7}{
      \myBAPP{Spaces of measurable functions}
         {submitted to \jRN[#1]{CollectM}}}
   \ITEE{#2}{pnX8}{
      \myBAPP{Normal systems over ANR's, rigid embeddings and nonseparable absorbing sets}
         {submitted to \jRN[#1]{ActaMSinES}}}
   \ITEE{#2}{pnX9}{
      \myBAPP{Borel structure of the spectrum of a closed operator}
         {submitted to \jRN[#1]{SM}}}
   \ITEE{#2}{pnX10}{
      \myBAPP{Central points and measures and dense subsets of compact metric spaces}
         {submitted to \jRN[#1]{TopMethNA}}}
   \ITEE{#2}{pnX11}{
      \myBAPP{Generalized absolute values and polar decompositions of a bounded operator}
         {submitted to \jRN[#1]{IEOT}.}}
   \ITEE{#2}{pnX12}{
      \myBAPP{Ultrametrics, extending of Lipschitz maps and nonexpansive selections}
         {accepted for publication in \jRN[#1]{HJM}}}
   \ITEE{#2}{pnX13}{
      \myBAPP{A note on ANR's}
         {submitted to \jRN[#1]{TopA}}}
   \ITEE{#2}{pnX14}{
      \myBAPP{Problem with almost everywhere equality}
         {submitted to \jRN[#1]{ArchM}}}
   \ITEE{#2}{pnX15}{
      \myBAPP{Universal valued Abelian groups}
         {submitted to \jRN[#1]{LNM}}}
   \ITEE{#2}{pnX16}{
      \myBAPP{Unitary equivalence and decompositions of finite systems of closed densely defined operators 
         in Hilbert spaces}{submitted to \jRN[#1]{DissM}}}
   }
